\documentclass{amsart}
\usepackage{amsmath,amssymb,amsthm,amsfonts,mathdots}
\usepackage{hyperref}
\hypersetup{colorlinks, citecolor=red,pdfstartview=FitH,pdfpagemode=None}
\usepackage{multicol}
\usepackage{mathrsfs}
\usepackage{graphicx,xcolor}
\usepackage{tikz}

  \tikzstyle{vertex} = [circle]
  \tikzstyle{edge} = [draw,thick,->]
  \tikzstyle{Redge} = [draw,thick,->,red]
  \tikzstyle{Gedge} = [draw,thick,->,green]
  \tikzstyle{Bedge} = [draw,thick,->,blue]
  \tikzstyle{arc} = [draw,dashed,-]
  \tikzstyle{label} = [font=\normalsize,rectangle,fill = white, text = red, fill opacity = 1, text opacity = 1, scale = 0.88]

\usepackage{etoolbox}
\let\bbordermatrix\bordermatrix
\patchcmd{\bbordermatrix}{8.75}{4.75}{}{}
\patchcmd{\bbordermatrix}{\left(}{\left[}{}{}
\patchcmd{\bbordermatrix}{\right)}{\right]}{}{}
\newcommand{\bmtx}{\bbordermatrix}  

\newtheorem{theorem}{Theorem}[section]
\newtheorem{lem}[theorem]{Lemma}

\newtheorem{thm}[theorem]{Theorem}

\theoremstyle{definition}
\newtheorem{definition}[theorem]{Definition}
\newtheorem{ex}[theorem]{Example}
\newtheorem{cor}[theorem]{Corollary}

\newtheorem{rem}[theorem]{Remark}

\numberwithin{equation}{section}

\newcommand\g{{\mathfrak g}}
\newcommand\s{{\mathfrak s}}
\renewcommand\l{{\mathfrak l}}

\newcommand{\mtx}[1]{\left[\,\begin{matrix} #1\end{matrix}\,\right ]}       
\newcommand{\smtx}[1]{\left[\begin{smallmatrix} #1\end{smallmatrix}\right ]}    
\newcommand{\rank}{\operatorname{rank}\;}       

\newcommand{\footer}[1]{{\def\thefootnote{}\footnotetext{#1}}}

\newcommand\R{\mathbb R}
\newcommand\C{\mathbb C}
\newcommand\N{\mathbb N}
\newcommand\Z{\mathbb Z}

\newcommand\F{\mathbb F}

\newcommand\diag{{\operatorname{diag}\,}}

\renewcommand\O{{\operatorname{O}}}

\newcommand\Lra {\Longrightarrow}

\newcommand\lra {\longrightarrow}

\newcommand\ra {\rightarrow}
\newcommand{\ul}[1]{\underline{#1}} 

\newcommand{\wt}[1]{\widetilde{#1}}

\newcommand{\mc}{\mathcal}

\newcommand{\Bel}{Belitski\u{\i}}
\newcommand{\Bform}{the Belitski\u{\i}'s canonical form}
\newcommand{\bform}{Belitski\u{\i}'s canonical form}
\newcommand{\Balg}{the Belitski\u{\i}'s algorithm}
\newcommand{\Border}{the Belitski\u{\i}'s order}
\newcommand{\beq}{\begin{equation}}
\newcommand{\eeq}{\end{equation}}
\newcommand{\BEQ}{\begin{equation*}}
\newcommand{\EEQ}{\end{equation*}}

\newcommand{\SIM}[1]{\overset{#1}{\sim}}

\newcommand{\red}[1]{{\color{red} #1 }}

\newcommand{\mpage}[2]{ \begin{minipage}{#1\textwidth} #2 \end{minipage}}  
\newcommand{\twopage}[4]{\mpage{#1}{#3}\mpage{#2}{#4}}

\begin{document}

\title{On Triangular similarity of   nilpotent  Triangular matrices}

\author[M. Tsai] {Ming-Cheng Tsai}
\address{General Education Center, 
Taipei University of Technology, Taipei 10608, Taiwan}
\email{mctsai2@mail.ntut.edu.tw}

\author[M. Bogale] {Meaza Bogale}
\address{Department of Mathematics and Statistics, Auburn University,
AL 36849--5310, USA}
 \email{mfb0012@tigermail.auburn.edu}

\author[H. Huang] {Huajun Huang*} 
\address{Department of Mathematics and Statistics, Auburn University,
AL 36849--5310, USA}
  \email{huanghu@auburn.edu}

\begin{abstract} Let $B_n$ (resp. $U_n$, $N_n$) be the set of $n\times n$ nonsingular (resp. unit, nilpotent) upper triangular matrices.  We use a novel approach to 
explore the $B_n$-similarity orbits in $N_n$. 
The \bform \ of $A\in N_n$ under $B_n$-similarity is in $QU_n$ where $Q$ is the subpermutation such that $A\in B_n QB_n$. Using graph representations and $U_n$-similarity actions stabilzing $QU_n$, we obtain new properties of \Bform s and present an efficient  algorithm  to find \Bform s in $N_n$. As consequences, we construct new \bform s  in all $N_n$'s,
list all \bform s  for $n=7, 8$, and show examples of 3-nilpotent
\bform s  in $N_n$ with arbitrary numbers of parameters up to $\O(n^2)$.
\end{abstract}



\maketitle

\section{Introduction}

Let $\F$ be a fixed field. Let $M_{m,n}$ (resp. $M_{n}$, $GL_{n}$) be the set of $m\times n$ (resp. $n\times n$,  $n\times n$ nonsingular) matrices over $\F$.
Let $B_n$ (resp. $U_n$, $N_n$) be the set of $n\times n$ nonsingular (resp. unit, nilpotent) upper triangular matrices, and $D_n$ the set of $n\times n$ nonsingular diagonal matrices, over $\F$.

The main goal of this paper is to describe the $B_n$-similarity orbits in $N_n$ through  \Bform s. 
We link a $B_n$-similarity orbit  to the corresponding $(B_n,B_n)$ double coset. 
Given $A\in N_n$, let $Q$ be the unique subpermutation such that $A\in B_n QB_n$.
The \bform \ of $A$ under $B_n$-similarity is in $QU_n$. We  improve \Balg\ to efficiently search for \Bform s using  graph representations  and graph operations on matrices in $QU_n$. As a consequence, all indecomposable \bform s for $n=7$ and $n=8$ are given,
which extends the works of D. Kobal \cite{Kobal2005} and Y. Chen et al \cite{Chen2016}. 
Moreover, we discover a   way to  obtain new indecomposable \bform s of any order $n$; we also present examples of $3$-nilpotent \bform s in $N_n$ with arbitrary number of  parameters up to $\O(n^2)$, which improves the $\O(n)$ result in \cite{Chen2016}.

The $B_n $-similarity orbits in $N_n $ is a special case of  the
 $\Lambda$-similarity matrix problem explored by V.  Sergeichuk in \cite{Sergeichuk2000}. Sergeichuk showed how the  $\Lambda$-similarity can be used to formulate the representations of quivers and  matrix problems \cite[Examples 1.1, 1.2]{Sergeichuk2000}, and presented \Balg \ to obtain  so called \Bform \ under  $\Lambda$-similarity. The strengthen Tame-Wild theorem for matrix problem (\cite[Theorem 3.1]{Sergeichuk2000}) and the existing classification on \Bform  s with two parameters  \cite{Chen2016} indicate that
the $B_n $-similarity problem on $N_n $ is of wild type. 
 
In 1978, M. Roitman discovered that if $\F$ is infinite,  the number of $B_n $-similarity orbits in $N_n $ is infinite for $n\ge 12$ \cite{Roitman1978}. 
D. Djokovi\'{c}  and J. Malzan improved the result to $n\ge 6$ in 1980 \cite{Djokovic1980}. 
D. Kobal   in 2005 listed all \bform s of the  $B_n$-similarity orbits  in $N_n$ for $n\le 5$  \cite{Kobal2005}. 
P. Thijsse  showed in 1997 that every upper triangular matrix  is $B_n $-similar to a generalized direct sum of irreducible blocks, and gave a  classification
of indecomposable (non-\Bel's) canonical forms    for $n\le 6$  \cite{Thijsse1997}. 
Besides,  Thijsse  showed that if an upper triangular matrix $A$   is nonderogatory or $A$ has Jordan block sizes no more than 2, then $A$ is $B_n $-similar to a generalized Jordan canonical form. 
In 2016, Y. Chen et al classified the indecomposable \bform s  for $n=6$  and for $n=7$ which admits   a parameter, and showed that   there exists an indecomposable \bform\   which admits
at least $\lfloor\frac{n}{2}\rfloor-2$ parameters   \cite{Chen2016}. 

When $\F=\C$ or $\R$, 
the  conjugacy orbits on nilpotent matrices or Lie algebra elements  were also  intensively investigated by Lie theorists and representation theorists. 
In the book \cite{Collingwood1993} of D. Collingwood and W. McGovern,  nilpotent $G$-orbits in semisimple Lie algebras $\g$  are bijectively corresponding to 
the $G$-orbits of the standard $\s\l_2$-triples, and are parameterized by weighted Dynkin diagrams. 
L. Fresse gave sufficient and necessary conditions for the intersection of  a nilpotent $GL_n$-orbit with $N_n$ to be a union of 
finitely many $B_n$-orbits \cite{Fresse2013}. 
A. Melnikov   described the $B_n$-orbits and their geometry on upper triangular $2$-nilpotent matrices  by link patterns in \cite{Melnikov2000, Melnikov2006, Melnikov2013}.  M. Boos and M. Reineke described the $B_n$-orbits and their  closure relations of all $2$-nilpotent matrices   \cite{Boos2012}. 
N. Barnea and  A. Melnikov described the Borel orbits of $2$-nilpotent elements in nilradicals for the symplectic algebra in 2017 \cite{Barnea2017}. 
M. Boos et al described the parabolic orbits of $2$-nilpotent  elements for classical groups  \cite{Boos2014, Boos2019}. 

The structure of this paper is as follows.
 
In Section \ref{sect: preliminary}, we review the classification and invariants of $(B_n,B_n)$ double cosets and \Balg\ for the $B_n$-similarity.
We show that \Bform\ of $A\in N_n$ is necessarily in $QU_n$ in which $Q$ is the subpermutation such that $A\in B_nQB_n$ (Theorem \ref{thm: Bform in right coset}). As a by product,   
we can construct  new \bform s $\smtx{A_1 &Q_{12}\\&A_2}$ when $A_1\in Q_1 U_{p}$ and $A_2\in Q_2U_{q}$ are \bform s and $\smtx{Q_1 &Q_{12}\\&Q_2}$ is a subpermutation in $N_{p+q}$ (Theorem \ref{thm: combine Bforms}). The criteria for $D_n$-similarity  is given in Theorem \ref{thm: D_n similarity}.  
Finally, every matrix in $B_nQB_n$ for a subpermutation $Q\in N_n$ can be transformed  via $B_n$-similarity to a matrix in $QU_n$, and this matrix can be transformed via
 elementary $U_n$-similarity operations (ESOs) stablizing $QU_n$ to  a matrix which is $D_n$-similar to \Bform\ (Theorem \ref{thm: U_n similar to Bform}). 

Section \ref{sect: graph operations} introduces the graph representations of matrices, and the graph operations corresponding to ESOs stabilizing $QU_n$. The graph operations visualize the 
$U_n$-similarity reduction process on $QU_n$ and help obtain \Bform s efficiently. 

Section \ref{sect: Bform properties} is devoted to explore the properties of \Bform\ through its graph. The graph of a \bform\ in $N_n$ with $m$ connected components and $N$ arcs has exactly $m$ indecomposable components and $N-n+m$ parameters (Theorem \ref{thm: Bform number of parameters}). 
Theorem \ref{thm: Bform places of parameters} determines the places of parameters in a \bform. 
Theorems \ref{thm: Bform no arcs} and \ref{thm: long chain deletion} prove that some entries in a \bform\ must be zero, and Theorem \ref{thm: Bform arcs} describes the possible places of nonzero entries. 
Finally, Theorem \ref{thm: Bform parameters} constructs 
indecomposable $3$-nilpotent \bform s with $r$ parameters 
for all $r\le \frac{1}{2} \lfloor\frac{n-2}{3}\rfloor(\lfloor\frac{n-2}{3}\rfloor-1)$ if $n\equiv 0, 2\mod 3$,
and $r\le \frac{1}{2} \lfloor\frac{n-2}{3}\rfloor(\lfloor\frac{n-2}{3}\rfloor-1)-1$ if $n\equiv 1\mod 3$. 

In Section \ref{sect: search Bforms},  we give an efficient graphical algorithm  to search for \Bform s based on Theorems \ref{thm: Bform arcs}  and \ref{thm: long chain deletion}. The algorithm significantly improves \Balg.
The indecomposable \bform s for $n=7$ is given in Theorem \ref{thm: Bform n=7}, and those for $n=8$ is given in Theorem \ref{thm: Bform n=8} and the Appendix. 
Examples of the algorithm, graph illustrations of Theorem \ref{thm: combine Bforms}, and connections to the $B_n$-similarity orbits of upper triangular matrices are also included
in this section.  

\section{Preliminary}\label{sect: preliminary}

\subsection{$B_n \times B_n $ action on $N_n $}

Given a subgroup $G$ of $GL_n$,
two matrices $A, C\in M_n$ are called {\em $G$-similar}, 
denoted by  $A\overset{G}{\sim} C$, 
if there exists  $B\in G$ such that $C=BAB^{-1}$. 
The   $A$ and $C$ are in the same $(B_n,B_n)$ double coset if there exist $B, B'\in B_n$ such that $C=BAB'$. 
The $B_n $-similarity orbit of $A\in M_n $ is contained in
the $(B_n , B_n )$ double coset of $A$:  
$$
\{BAB^{-1}\mid B\in B_n \}   \subseteq   B_n AB_n:= \{B AB' \mid B, B' \in B_n \}.
$$
The $(B_n , B_n )$ double cosets on $M_n $ are well classified  as an extension of both the Bruhat decomposition in semisimple Lie groups and  Gelfand-Naimark decomposition in matrix theory. We review the results here. 

\begin{definition}
A matrix $Q\in M_{m,n}$ is called a {\em subpermutation} if each of the rows and columns of $Q$ has at most one nonzero entry, which equals $1$. 
\end{definition}

Let $[n]:=\{1,2,\ldots,n\}$. 
Given $i,j\in [n]$, let $E_{i,j}^{(n)}\in M_n$ (or $E_{ij}^{(n)}$ for simplicity) be the matrix that has  1 on the $(i,j)$ entry and 0's elsewhere, and let $e_i^{(n)}\in\F^{n}$ be the vector that has 1 on the $i$th entry and 0's elsewhere. They are abbreviated as $E_{i,j}$ (or $E_{ij}$ for simplicity) and $e_i$, respectively, if the size $n$ is clear. 
Every subpermutation $Q\in M_n$ can be determined by a bijective map $\sigma: I\to \sigma(I)$ between two subsets $I$ and $\sigma(I)$ of $[n]$ of the same cardinality:
\begin{equation}\label{subpermutation}
Q=\sum_{i\in I} E_{i,\sigma(i)};\qquad
Q:=0\text{ \  if \ }I=\emptyset.
\end{equation}

Given $A\in M_n $ and $I,J\subseteq [n]$, let $A[I,J]$ denote the submatrix of $A$ with rows indexed by $I$ and columns indexed by $J$.  
Moreover, given $i, j\in [n]$, let 
\begin{equation}
r^{i,j}(A):= \rank A[[n]\setminus[n-i], [j]]=\rank A[\{n-i+1,\ldots,n\}, \{1,\ldots,j\}]
\end{equation}
be the rank of the  lower left $i\times j$ submatrix of $A$; define $r^{0,j}(A)=r^{i,0}(A):=0$. 

The following characterization of $(B_n,B_n)$ double cosets on $M_n$ is classical. Analogic double coset results on $GL_n$ can be found in \cite[Theorem 3.5.14]{Horn2013}.

\begin{lem}\label{thm: double coset}
The $(B_n, B_n)$ double coset of $A\in M_n$ is completely determined by the set of invariants:
\begin{equation}
 \{r^{i,j}(A): i,j\in [n]\}.
\end{equation}
There is a unique subpermutation $Q\in M_n$ such that $A\in B_n Q B_n$. The  entries of $Q=[q_{ij}]$ are determined by:
\begin{equation}\label{rank to subpermutation entries}
q_{n-i+1,j}=r^{i,j}(A)-r^{i-1,j}(A)-r^{i,j-1}(A)+r^{i-1,j-1}(A),\quad i,j\in [n]. 
\end{equation}
\end{lem}

\begin{proof}
Given arbitrary $B, B'\in B_n$ and $i,j\in [n]$, we look at $BAB'$ from the following partitions: 
\begin{eqnarray*}
BAB' &=& \bmtx{&\small{n-i} &i \cr n-i &B_{11} &B_{12}\cr i &0 &B_{22}}\ \bmtx{ &j &n-j\cr n-i &A_{11} &A_{12}\cr i &A_{21} &A_{22}}\ \bmtx{&j &n-j\cr j &B_{11}' &B_{12}'\cr n-j &0 &B_{22}'}
\\
&=& \bmtx{&j &n-j\cr n-i &\star &\star \cr i &B_{22}A_{21}B_{11}' &\star}.
\end{eqnarray*}
Both $B_{22}\in M_{i}$ and $B_{11}'\in M_{j}$ are nonsingular. Therefore, 
$r^{i,j} (BAB') =r^{i,j}(A).$

Next, we illustrate how to transform  $A=[a_{ij}]\in M_n$ to a subpermutation $Q$ through elemantary row and column operations associated with  muliplications of matrices in $B_n$. 
\begin{enumerate}
\item
Start from the last row of $A$. If it is a zero row, we are done for the row. Otherwise, let $\sigma(n)\in [n]$ such that $a_{n \sigma(n)}$ is the first nonzero entry of the row. For each $j \in [n]\setminus[\sigma(n)]$, add a multiple $(-a_{nj}/a_{n\sigma(n)})$ of   the  $\sigma(n)$th column of $A$ to the $j$th column of $A$. These elementary column operations   result in multiplying $A$ from the right by a matrix $B_{(1)}'\in B_n$. Denote $A_1'=AB_{(1)}'$.  Then for each $i\in [n-1]$, 
add a multiple $(-a_{i\sigma(n)}/a_{n\sigma(n)})$  of the $n$th row of $A_1'$  to the $i$th row of $A_1'$. These elementary row operations result in 
multiplying $A_1'$ from the left by a matrix $B_{(1)}\in B_n$. Denote a new matrix $A_1=[a_{ij}^{(1)}]=B_{(1)}A_1'=B_{(1)}AB_{(1)}'.$ Then $a_{n\sigma(n)}^{(1)}=a_{n\sigma(n)}$ is the only nonzero entry of its row and column in $A_1$.

\item
Repeat the same strategy on the other rows of the new matrix in the reversing row order until all rows are done. 
\end{enumerate}
The above process produces a matrix $Q '=B_*AB_* ''$ in which each of the rows and columns has at most one nonzero entry. By multiplying an appropriate nonsingular diagonal matrix $D'$ from the right, we get   a subpermutation $Q=B_*AB_*''D'=B_*AB_*'$ for some $B_*, B_* '\in B_n$.  

Clearly, $r^{i,j}(Q)=r^{i,j}(A)$ for $i,j\in [n]\cup\{0\}$. Moveover, given $i,j\in[n]$, 
$Q[[n]\setminus [n-i], [j]]$  has exactly one of the following forms ($k\in [j-1]$, $l\in [i-1]$):
\begin{align*}
\mtx{0 &1\\ Q[[n]\setminus [n-i+1], [j-1]] &0}, &\quad &
\mtx{(e_k^{(j-1)})^T &0\\ Q[[n]\setminus [n-i+1], [j-1]] &0},
\\
\mtx{0 &0\\ Q[[n]\setminus [n-i+1], [j-1]] &e_l^{(i-1)}}, &\quad &
\mtx{0 &0\\ Q[[n]\setminus [n-i+1], [j-1]] &0}.
\end{align*}
In all cases, the entries of subpermutation $Q=[q_{ij}]$ can be obtained by:
\begin{equation}
q_{n-i+1,j}=r^{i,j}(Q)-r^{i-1,j}(Q)-r^{i,j-1}(Q)+r^{i-1,j-1}(Q),\quad i,j\in [n]. 
\end{equation}
Therefore, the set of invariants  $\{r^{i,j}(A): i,j\in [n]\}$ completely determines the unique subpermutation $Q$ and  the corresponding  $(B_n, B_n)$ double coset of $A$. 
\end{proof}

If two  matrices are similar and in the same $(B_n, B_n)$ double coset,
are they necessarily $B_n$-similar? The answer is no.  

\begin{ex} 
Let $A=\smtx{0 & 1 & 1 & 0 & 0 \cr  & 0 &0 &0 & 1 \cr  &  & 0 & 1 & 0 \cr   &  &  & 0 & 0 \cr  &  &  &  & 0 \cr}$ and $B=\smtx{0 & 1 & 0 & 0 & 0 \cr  & 0 &0 &0 & 1 \cr  &  & 0 & 1 & 0 \cr   &  &  & 0 & 0 \cr  &  &  &  & 0 \cr}$. Both $A$ and $B$ have the only eigenvalue 0, and $\rank (A^m)=\rank (B^m)$ for all $m\in\Z^+$. So $A$ and $B$ 
are similar. They are also in the same $(B_n, B_n)$  double coset represented by the subpermutation $B$.  However, $A$ and $B$ are not $B_n$-similar \cite[Theorem 2]{Kobal2005}. 
\end{ex}

The $(B_n, B_n)$ double coset provides a  good direction to explore the  $B_n$-similarity orbits it includes. 
Suppose $A=BQ B'\in N_n$ where $B, B'\in B_n$ and $Q\in N_n$ is a subpermutation. Then $A\overset{B_n}{\sim} Q B'B$. 
Write $B'B=DU$ for $D\in D_n$ and $U\in U_n$. 
Since $Q\in N_n$, there exists $D'\in D_n$ such that $D'QD(D')^{-1}=Q$. Then 
$$A\SIM{B_n} QB'B=QDU\SIM{D_n} D' QDU (D')^{-1}=QD'U(D')^{-1}\in QU_n.$$
The coset $Q U_n$ takes the following form.

\begin{lem}\label{thm: subperm coset}
Suppose $Q=\sum_{i\in I} E_{i,\sigma(i)}\in M_n$ is a subpermutation. Then $A=[a_{ij}]\in Q U_n$ if and only if $A$ meets the following conditions:
\begin{enumerate}
\item
 $A$ and $Q$ have the same places of nonzero rows indexed by $I$;
\item
$A$ and $Q$ have the same places and values of the first nonzero entry in each nonzero row; precisely,
for each $i\in I$, $a_{i,\sigma(i)}=1$ is the first nonzero entry of the $i$th row of $A$.
\end{enumerate}
\end{lem}

The proof can be done by direct computation.

\subsection{The \Bel’s algorithm for the $B_n$-similarity in $N_n$}\label{sect: Balg}

 V.  Sergeichuk presented {\em \Balg}  to find a canonical form, {\em \Bform}, for the  $\Lambda$-similarity matrix problem 
  \cite{Sergeichuk2000}. On the $B_n$-similarity of $A\in N_n$, the algorithm can be described as follows. 

\begin{enumerate}
\item \label{Balg order}
List the matrix entry positions above the diagonal in a reversal row lexicographical order ``$\prec$'' called {\em \Border}: 
\begin{equation}\label{Border}
(n-1,n)\prec (n-2,n-1)\prec  (n-2, n)\prec  (n-3,n-2)\prec \cdots\prec (1,n).
\end{equation}
The strictly upper triangular entries  will be normalized 
through $B_n$-similarity  one-by-one in this order.

\item (Normalizing the first entry)
Let $(A^{(0)}, B^{(0)}):=(A,B_n)$. 
Find $A^{(1)}$  in the $B^{(0)}$-similarity orbit of $A^{(0)}=[a_{ij}]$ such that the $(n-1,n)$ entry of $A^{(1)}$   is either 0 or 1. 
For example,
$$
A^{(1)}:=
\begin{cases} 
A &\text{if \ } a_{n-1,n}=0,\\ (I_{n-1}\oplus [a_{n-1,n}])A(I_{n-1}\oplus [a_{n-1,n}])^{-1} &\text{if \ } a_{n-1,n}\ne 0.
\end{cases}
$$
Denote  the group
$$B^{(1)}:=\{g\in B^{(0)}\mid gA^{(1)}g^{-1} \text{ fixes   the value of the $(n-1,n)$ entry of $A^{(1)}$}\}.$$ 

\item  (Normalizing the consequent entries)
Suppose $(A^{(k)},B^{(k)})$ has been determined, and the group $B^{(k)}$ fixes the first $k$ entries of $A^{(k)}=[a_{ij}']$ in \Border.  Let $(p,q)$ be the $(k+1)$th entry position.
There are three situations for the $(p,q)$ entry of matrices $C=[c_{i,j}]$ in the $B^{(k)}$-similarity orbit of $A^{(k)}$:
\begin{enumerate}
\item  $c_{p,q}$ is always $0$, or $c_{p,q}$ could take any value of $\F$: we find $A^{(k+1)}=[a_{i,j}'']\overset{B^{(k)}}{\sim}A^{(k)}$ such that $a_{p,q}''=0$;
\item $c_{p,q}$ could take any  value of $\F\setminus\{0\}$: we find $A^{(k+1)}=[a_{i,j}'']\overset{B^{(k)}}{\sim}A^{(k)}$ such that $a_{p,q}''=1$;
\item otherwise, $c_{p,q}\equiv\lambda$ for a fixed $\lambda\in\F\setminus\{0\}$: we   choose $A^{(k+1)}=A^{(k)}$ with $a_{p,q}''=\lambda$.
\end{enumerate}
Let $B^{(k+1)}$ denote the subgroup of $B^{(k)}$ that  fixes the $(k+1)$th entry value as well as the first $k$ entry values of $A^{(k+1)}$.

\item
 Repeat the preceding step until the last position in \Border \ is reached.
Denote the last pair
$(A^\infty, B^\infty)$. 
 The matrix $A^\infty$  is called {\em the Belitski\u{\i}'s canonical form of $A$ under the $B_n$-similarity}. 

\end{enumerate}

The above algorithm shows that each upper triangular entry of \Bform \ $A^\infty$ is $0$ or $1$ or a {\em parameter} $\lambda$ in which
different $\lambda$ values correspond to different $B_n$-similarity orbits. This property is similar to that of a Jordan canonical form. Moreover, \Bform \ $A^\infty$ has the following connection to the subpermutation $Q$ in the $(B_n,B_n)$ double coset of $A$ and $A^\infty$.

\begin{thm}\label{thm: Bform in right coset}
Given a  Belitski\u{\i}'s canonical form $A\in N_n$, if $A\in B_n Q B_n$ in which $Q$ is a subpermutation, then $A\in Q U_n$. 
\end{thm}

\begin{proof}
The proof  is done by induction on $n$. $n=1$ is obviously true. Suppose the statement holds for all $n<m$. 
Given  $A\in B_mQ B_m$ where $Q\in N_{m}$ is a subpermutation, write $A=\mtx{0 &a^T\\ &A_1}$ for $A_1\in N_{m-1}$ and $a\in\F^{m-1}$.
By \Balg, $A_1$ is  a Belitski\u{\i}'s canonical form in $N_{m-1}$. 
Write $Q=\mtx{0 &b^T\\ &Q_1}$ in which $Q_1$ is a subpermutation in $N_{m-1}$ and $b\in\F^{m-1}$.
Then $A_1\in B_{m-1}Q_1 B_{m-1}$. So by induction hypothesis  $A_1= Q_1 \hat U$ for $\hat U\in U_{m-1}$. 
\begin{enumerate}
\item If $a=0$, then Lemma \ref{thm: double coset} implies that $Q=\mtx{0 &0\\&Q_1}$. Hence 
$$A=\mtx{0 &0\\&Q_1}\mtx{1 &0\\&\hat U}\in Q U_m.$$

\item If $a\ne 0$, let $A=[a_{ij}]$ and let $a_{1q}$ ($q\in\{2,\ldots,m\}$) be the leading nonzero entry in the first row of $A$.
Then
$$A\SIM{B_{m}}\mtx{a_{1q} &0\\ &I_{m-1}}^{-1}A\mtx{a_{1q} &0\\ &I_{m-1}}=\mtx{0 & a_{1q}^{-1} a^T\\ &A_1}$$ 
in which the last matrix
has the leading entry $1$ on the $(1,q)$ position. 
By \Balg \ $a_{1q}=1$.
  We claim that 
there is no $p\in\{2,\ldots,m\}$ such that $a_{pq}$ is the leading nonzero entry of the $p$th row of $A$  (i.e. the $(p-1)$th row of $A_1$). 
Otherwise,
\begin{eqnarray*}
A
&\SIM{B_m} &
\mtx{1 &\frac{1}{a_{pq}}(e_{p-1}^{(m-1)})^T\\ &I_{m-1}}^{-1}\mtx{0 &a^T\\ &A_1} \mtx{1 &\frac{1}{a_{pq}}(e_{p-1}^{(m-1)})^T\\ &I_{m-1}}
\\
&=& \mtx{0 &a^T-\frac{1}{a_{pq}}(e_{p-1}^{(m-1)})^TA_1\\ &A_1} 
\end{eqnarray*} 
where the first row of the last matrix has at least  $q$ leading zeros;  contradicting \Balg.   
By Lemma \ref{thm: subperm coset},  the $(q-1)$th column of $Q_1$ is zero.
Using \eqref{rank to subpermutation entries}, 
we have $Q=\mtx{0 &(e_{q-1}^{(m-1)})^T\\ &Q_1}$.
Let $\hat U(a^T)$ denote the matrix obtained by replacing the $(q-1)$th  row of $\hat U$ by $a^T$. 
Then $\hat U(a^T)\in U_{m-1}$ and 
$$A=\mtx{0 &(e_{q-1}^{(m-1)})^T\\ &Q_1}\mtx{1 &0\\ &\hat U(a^T)}\in Q U_m.$$
\end{enumerate}
Overall, the statement holds for $n=m$ and the induction process is completed.
\end{proof}

\begin{rem}
A \bform \  needs not be in $U_nQ$ or $B_nQ$. See the  examples in
Theorems \ref{thm: Bform n le 6}, \ref{thm: Bform n=7}, and \ref{thm: Bform n=8}.
\end{rem}

The direct sums of \bform s are obviously  \bform s. Moreover, Theorem \ref{thm: Bform in right coset} implies a way to combine \bform s together through certain subpermutations
to form a new \bform, as shown below.

\begin{thm}\label{thm: combine Bforms}
Suppose $A_1\in N_p$ and $A_2\in N_q$ are \bform s, in which $A_1\in Q_1U_p$ and $A_2\in Q_2U_q$ for subpermutations $Q_1\in N_p$ and $Q_2\in N_q$. 
If $Q_{12}\in M_{p,q}$  such that $\mtx{Q_1 &Q_{12}\\0 &Q_2}$ is a subpermutation,
then $\mtx{A_1 &Q_{12}\\0 &A_2}$ is a \bform\ in $N_{p+q}$.
\end{thm}

\begin{proof}
Let $A=\mtx{A_{11} &A_{12}\\0 &A_{22}}$ ($A_{11}\in N_p$) be \Bform\ of  $A':=\mtx{A_1 &Q_{12}\\0 &A_2}$. 
Then $A_{22}=A_2$ by \Balg.

Let $Q:=\mtx{Q_1 &Q_{12}\\0 &Q_2}$. 
Write $A_1=Q_1U'$ and $A_2=Q_2U''$ for $U'\in U_{p}$ and $U''\in U_{q}$. 
Then the nonzero rows of $A'=\mtx{Q_1U' &Q_{12}\\0 &Q_2U''}$ have
the same places and values (i.e., $1$) of leading nonzero entries as  the nonzero rows of $Q$ do.
Therefore, 
 $A'\in QU_{p+q}$ by Lemma \ref{thm: subperm coset}, and $A\in QU_{p+q}$ by Theorem \ref{thm: Bform in right coset}. 
 
 Now consider $A_{11}$ and $A_{12}$.
One one hand, 
each nonzero entry of the subpermutation $Q_{12}$ 
equals the corresponding row leading nonzero entry of  $A_{12}$. 
On the other hand, $A'\SIM{B_{p+q}} A$ implies that $A_{1}\SIM{B_{p}} A_{11}$;
 $A_{11}$ cannot be further reduced from \Bform\ $A_{1}$ in \Balg. 
 Therefore,  $A_{11}=A_1$ and $A_{12}=Q_{12}$ by \Balg, so that  $A=\mtx{A_1 &Q_{12}\\0 &A_2}$ is a \bform.
\end{proof}

\begin{rem}
In Theorem \ref{thm: combine Bforms}, the form of \Bform\ $\mtx{A_1 &Q_{12}\\0 &A_2}$ 
could have more parameters in nonzero entries of $A_1$ and $A_2$ than in the original \bform s $A_1$ and $A_2$. For an example, see the case
$A_1=A_2=\smtx{0 &1 &1 &0\\ &0&0&0\\&&0&1\\&&&0}$ and $Q_{12}=\smtx{0 &0 &0 &0\\1 &0 &0 &0\\0 &0 &0 &0\\0 &0 &1 &0}$ 
in Example \ref{ex: combine bforms}. 
\end{rem}

\subsection{$B_n$, $D_n$, and $U_n$ similarities}

On the group level, $B_n=D_n\ltimes U_n$. Two matrices
  $A\SIM{B_n} C$ if and only if $C=BAB^{-1}$ for  $B\in B_n$ and $B=UD$ such that $D\in D_n$ and $U\in U_n$, so that
 $A\SIM{D_n} DAD^{-1}\SIM{U_n} C$. The $D_n$-similarity  on  $M_n$ is easy to classify.


In this paper, $A\in M_n$ is called {\em indecomposable} if no permutation matrix $P\in M_n$ satisfies that $PAP^T$ can be written as a direct sum of two proper principal submatrices. The notation is different  from that in \cite{Chen2016}, but they are identical when referring to an
indecomposable \bform.

Given $A\in M_n$ and $i,j\in [n]$, let us define 
\begin{equation}
f_{ij}(A):=\begin{cases}
a_{ij} &\text{if \ } a_{ij}\ne 0,\\
\frac{1}{a_{ji}} &\text{if \ } a_{ij}=0 \text{ but } a_{ji}\ne 0,\\
0&\text{if \ } a_{ij}=a_{ji}=0.
\end{cases}
\end{equation}

\begin{thm}\label{thm: D_n similarity}
Two matrices $A=[a_{ij}], C=[c_{ij}]\in M_n$ have $A\SIM{D_n} C$ if and only if the following two conditions hold:
\begin{enumerate}
\item $A$ and $C$ have the same places of nonzero entries, namely, $a_{ij}\ne 0$ if and only if $c_{ij}\ne 0$; and 
\item for every
sequence $(i_1,\ldots,i_p)$ of distinct elements in $[n]$ such that 
at least one of
$a_{i_{k}i_{k+1}}$ and $a_{i_{k+1}i_{k}}$ is nonzero for each $k\in [p]$ (let $i_{p+1}:=i_1$), we have the identity
\begin{equation}\label{D_n similarity relation}
f_{i_1i_2}(A)\cdots f_{i_{p-1}i_p}(A)f_{i_pi_1}(A)=
f_{i_1i_2}(C)\cdots f_{i_{p-1}i_p}(C)f_{i_pi_1}(C).
\end{equation}
\end{enumerate}
\end{thm}

\begin{proof}
Suppose $C=DAD^{-1}$ where $D=\diag(d_1,\ldots,d_n)$ is nonsingular. Then $c_{ij}=\frac{d_i}{d_j} a_{ij}$ for $i,j\in [n]$. Conditions
(1) and (2) in the theorem obviously hold.

Conversely, we use induction on $n$ to prove that  (1) and (2)  imply $A\SIM{D_n} C$.  $n=1$ is  true.
Suppose the claim  holds for  all cases of $n<m$. Now for $n=m$, let $A, C\in M_n$  satisfy (1) and (2). 
If $A$ is not indecomposable, then there is a permutation matrix $P$ such that $PAP^T$ and
$PCP^T$ are direct sums of respective proper principal submatrices. So
by induction hypothesis $PAP^T\SIM{D_n}PCP^T$ and $A\SIM{D_n}C$.
Otherwise, $A$ is indecomposable. We find $d_1,\ldots, d_n\in\F\setminus\{0\}$  as follows such that $c_{ij}=\frac{d_i}{d_j} a_{ij}$ for $i,j\in [n]$. 
Let $S_0:=\{1\}$ and $d_1:=1$.

\begin{enumerate}
\item
Since $A$ is indecomposable, there are $j\in [n]\setminus S_0$ such that $a_{1j}\ne 0$ or $a_{j1}\ne 0$, in which we define 
\begin{equation}\label{d_j}
d_j:=\begin{cases} d_1\frac{a_{1j}}{c_{1j}} &\text{if \ } a_{1j}\ne 0,
\\
d_1\frac{c_{j1}}{a_{j1}} &\text{if \ } a_{1j}=0,\ a_{j1}\ne 0.
\end{cases}
\end{equation} 
In the case $a_{1j}\ne 0$ and $a_{j1}\ne 0$, \eqref{D_n similarity relation} gives $a_{1j}a_{j1}=c_{1j}c_{j1}$ so that 
the $d_j$ defined by \eqref{d_j} satisfies both
$c_{1j}=\frac{d_1}{d_j}a_{1j}$ and $c_{j1}=\frac{d_j}{d_1} a_{j1}$. 
Let 
\begin{equation*}
S_1:=S_0\cup\{j\in [n]\setminus S_0: a_{1j}\ne 0 \text{ or } a_{j1}\ne 0\}.
\end{equation*} 
Then $S_1\supsetneq S_0$ and $c_{ij}=\frac{d_i}{d_j} a_{ij}$ for $i,j\in S_1$. 

\item
If $S_1\ne [n]$,
then $A$ being indecomposable  implies that  
$a_{ij}\ne 0$ or $a_{ji}\ne 0$ for some $(i,j)\in S_1\times ([n]\setminus S_1)$, in which we define
\begin{equation}
d_j:=\begin{cases} d_i\frac{a_{ij}}{c_{ij}} &\text{if \ } a_{ij}\ne 0,
\\
d_i\frac{c_{ji}}{a_{ji}} &\text{if \ } a_{ij}=0,\ a_{ji}\ne 0.
\end{cases}
\end{equation} 
 Let
\begin{equation*}
S_2:=S_1\cup\{j\in [n]\setminus S_1:  a_{ij}\ne 0 \text{ or } a_{ji}\ne 0 \text{ for some }  i\in S_1\}.
\end{equation*} 
Then $S_2\supsetneq S_1$ and $c_{ij}=\frac{d_i}{d_j} a_{ij}$ for $i,j\in S_2$ by \eqref{D_n similarity relation}. 

\item
Repeat the process until we reach $S_m=[n]$, where all $d_j$ for $j\in [n]$ are well-defined. 
Let $D:=\diag(d_1,\ldots,d_n)$ then $C=DAD^{-1}$ as desired.
\qedhere 
\end{enumerate}
\end{proof}

Theorem \ref{thm: D_n similarity} shows that:
if  $A\in M_n$ is transformed via a $B_n$-similarity action to  $C\in M_n$, the zero places of $C$ are determined by the associated $U_n$-similarity transformation. 
The identities \eqref{D_n similarity relation} in Theorem \ref{thm: D_n similarity} will also be used to determine the places of parameters in a \bform. 

The matrix group $U_n$ is generated by 
\begin{equation}
\{I_n+\lambda E_{pq}: \lambda\in\F,\ p,q\in [n],\ p<q\}.
\end{equation}

\begin{definition}\label{def: ESO}
Given $\lambda\in\F$, $(p,q)\in [n]\times [n]$ and $p<q$,  we define   an {\em  elementary $U_n$-similarity operation (ESO)} to be the function $\O_{p,q}^{\lambda}: N_n\to N_n$ such that
for $A =[a_{ij}]\in N_n$: 
\begin{eqnarray}
\O_{p,q}^{\lambda}(A) 
&:=&
(I_n+\lambda E_{pq})A(I_n+\lambda E_{pq})^{-1}
\notag\\
&=& 
(I_n+\lambda E_{pq})(\sum_{i,j=1}^{n} a_{ij}E_{ij})(I_n-\lambda E_{pq})
\notag
\\ \label{elementary similarity}
&=& A+\sum_{\substack{j\in [n]\\ a_{qj}\ne 0}}\lambda a_{qj}E_{pj}-\sum_{\substack{i\in [n]\\a_{ip}\ne 0}}  \lambda a_{ip}E_{iq}.
\end{eqnarray}
Each $\O_{p,q}^{\lambda}$ is also called an {\em $\O_{p,q}$-operation}.
\end{definition}

The ESOs will be described
by graph operations in Section 3.

\begin{lem}\label{thm: U by elementary matrices}
Given $U\in U_n$, write $U=I_n+\sum_{k=1}^{m} u_{i_kj_k} E_{i_kj_k}$ where 
 $(i_1,j_1)\prec (i_2,j_2)\prec\cdots\prec (i_m,j_m)$  in \Border\ \eqref{Border}. 
Then 
\begin{equation}\label{elementary matrices generates others}
U=(I_n+u_{i_1j_1}E_{i_1j_1})\cdots (I_n+u_{i_mj_m}E_{i_mj_m}).
\end{equation}
\end{lem}

\begin{proof}
Left multiply $(I_n+u_{i_1j_1}E_{i_1j_1})^{-1}$ onto $U$. The matrix
$(I_n+u_{i_1j_1}E_{i_1j_1})^{-1}U=(I_n-u_{i_1j_1}E_{i_1j_1})U$ is the one that eliminates the $(i_1,j_1)$ entry of $U$.
Keep left multiplying $(I_n+u_{i_2j_2}E_{i_2j_2})^{-1},\ldots,(I_n+u_{i_mj_m}E_{i_mj_m})^{-1}$ in order.
We will have
$$(I_n+u_{i_mj_m}E_{i_mj_m})^{-1}\cdots (I_n+u_{i_1j_1}E_{i_1j_1})^{-1}U=I_n.$$
So \eqref{elementary matrices generates others} holds. 
\end{proof}

\begin{rem}
Given $U\in U_n$, if we write $U^{-1}=I_n-\sum_{k=1}^{m} u_{i_kj_k}' E_{i_kj_k}$ where $(i_1,j_1)\prec (i_2,j_2)\prec\cdots\prec (i_m,j_m)$, then \eqref{elementary matrices generates others} implies that
$$
U^{-1}=(I_n-u_{i_1j_1}' E_{i_1j_1})\cdots (I_n-u_{i_mj_m}' E_{i_mj_m})
$$
so that
\begin{equation}
U=(I_n+u_{i_mj_m}' E_{i_mj_m}) \cdots (I_n+u_{i_1j_1}' E_{i_1j_1}).
\end{equation}
\end{rem}

\begin{lem}\label{thm: group generated by elements}
Let $S\subseteq \{(i,j)\in [n]\times[n]: i<j\}$ such that 
\begin{equation}\label{U_S}
U_S:=\{I_n+\sum_{(i,j)\in S} a_{ij}E_{ij}: a_{ij}\in \F\}
\end{equation} 
is a subgroup of $U_n$.
Then $U_S$ is generated by $\{I_n+\lambda E_{ij}: (i,j)\in S,\lambda\in\F\}$, and each element of
$U_S$ can be written as a product of  no more than $|S|$ elements in $\{I_n+\lambda E_{ij}: (i,j)\in S,\lambda\in\F\}$.
\end{lem}

\begin{proof}
It is a direct consequence of Lemma \ref{thm: U by elementary matrices}.
\end{proof}

Given a subpermutation $Q$, the coset $Q U_n$ is not closed under the $U_n$-similarity. However,
the following result indicates that   $U_n$-similar  matrices in $Q U_n$ 
can be transformed to each other via finitely many ESOs stabilizing $Q U_n$.

\begin{thm}\label{thm: ESO on coset}
Let $Q\in N_n$ be a subpermutation.  Let $A, C\in Q U_n$ such that $A\SIM{U_n}C$. 
Then there exist a sequence of ESOs $\{\O_{i_1,j_1}^{\lambda_1},\ldots,\O_{i_m,j_m}^{\lambda_m}\}$, 
$\lambda_k\in\F$ and $1\le i_k<j_k\le n$ for $k\in [m]$, such that the followings conditions hold: 
\begin{enumerate}
\item   $(i_1,j_1)\prec (i_2,j_2)\prec\cdots\prec (i_m,j_m)$ in \Border\ \eqref{Border}.
\item Let $A_0:=A$ and for $k\in [m]$: 
\begin{equation}
A_{k}:=\O_{i_k,j_k}^{\lambda_k}(A_{k-1})=(I_n+\lambda_k E_{i_kj_k})A_{k-1}(I_n+\lambda_k E_{i_kj_k})^{-1}.
\end{equation}
Then $A_0, A_1,\ldots,A_m\in QU_n$ and $A_m=C$. 
\end{enumerate}
\end{thm}

\begin{proof} Let $Q=\sum_{i\in I} E_{i,\sigma(i)}$ as in \eqref{subpermutation}. 
Let $A=Q U'$ and $C=UAU^{-1}=UQU'U^{-1}$ for $U, U'\in U_n$. 
Write $U=I_n+[u_{ij}]$ where $[u_{ij}]\in N_n$. By direct computation, $C\in QU_n$ if and only if  $UQ\in QU_n$, if and only if the nonzero $u_{ij}$ entries have the pairs $(i,j)$ in the set
\begin{equation}\label{index set of U_n stabilizing QU_n}
S_Q:=\{(i,j)\in I\times I: i<j,\sigma(i)<\sigma(j)\}\cup \{(i,j)\in [n]\times ([n]\setminus I): i<j\}.
\end{equation}
Therefore, the group $\{T\in U_n: TAT^{-1}\in QU_n\}=U_{S_Q}$ which is generated by  $\{I_n+\lambda E_{ij}: (i,j)\in S_Q,\ \lambda\in\F\}$ according to Lemma \ref{thm: group generated by elements}. Moreover, $U^{-1}\in U_{S_Q}$. If we write  
$$
U^{-1}=I_n-\sum_{k=1}^{m} \lambda_{k} E_{i_kj_k},
$$ 
in which $\lambda_{k}\in\F\setminus\{0\}$, $(i_k,j_k)\in S_Q$ and $(i_1,j_1)\prec \cdots\prec (i_m,j_m)$ in \Border,
then by Lemma \ref{thm: U by elementary matrices},
$U^{-1}=(I_n-\lambda_1 E_{i_1j_1})\cdots (I_n-\lambda_m E_{i_mj_m})$ and
$$
C =(I_n+\lambda_m E_{i_mj_m})\cdots (I_n+\lambda_1 E_{i_1j_1})A (I_n+\lambda_1 E_{i_1j_1})^{-1}\cdots (I_n+\lambda_m E_{i_mj_m})^{-1}.
$$
So  Theorem \ref{thm: ESO on coset} (1) and (2) are proved.
\end{proof}

\begin{rem} 
Theorem \ref{thm: ESO on coset} also holds if we replace condition (1) by the condition: $(i_1,j_1)\succ (i_2,j_2)\succ\cdots\succ (i_m,j_m)$ in \Border\ \eqref{Border}. 
\end{rem}

\begin{thm}\label{thm: U_n similar to Bform}
If $A\in Q U_n$ and $Q\in N_n$ is a  subpermutation, then $A$ can be transformed via a finite number of ESOs stabilizing $Q U_n$
to a matrix $\widetilde{A^\infty}\in Q U_n$ which is $D_n$-similar to \Bform \ $A^\infty\in QU_n$.
\end{thm}

\begin{proof}
Since $B_n=D_n\ltimes U_n$, there exist $D\in D_n$ and $U\in U_n$ such that $A^\infty=(DU)A(DU)^{-1}=D(UAU^{-1})D^{-1}$. Let
$\widetilde{A^\infty}:=UAU^{-1}$. 
We first prove that $\widetilde{A^\infty}\in QU_n$. Notice that $A^\infty\in QU_n$ by Theorem \ref{thm: Bform in right coset}, and
$\widetilde{A^\infty}$ and $A^\infty=D\widetilde{A^\infty}D^{-1}$ have the same places of nonzero entries by Theorem \ref{thm: D_n similarity}.
Using Lemma \ref{thm: subperm coset}, it suffices to show that the leading nonzero entry of each nonzero row of $\widetilde{A^\infty}$ equals $1$. 
Let $R_i(C)$ denote the $i$th row of a matrix $C$.  By $A\in QU_n$, we have $AU^{-1}\in QU_n$ so that
all nonzero rows $R_i(AU^{-1})$ have distinct places of leading nonzero entries $1$. 
Let $U:=[u_{i,j}]$. 
Then $\widetilde{A^\infty}=U(AU^{-1})$ implies that for $i\in[n]$:
$$R_i(\widetilde{A^\infty})=R_i(AU^{-1})+u_{i,i+1}R_{i+1}(AU^{-1})+\cdots+u_{i,n}R_{n}(AU^{-1}).$$
Suppose $R_i(\widetilde{A^\infty})$ is a nonzero row for a given $i$. Then $R_i(\widetilde{A^\infty})$, $R_i(A^\infty)$, and $R_i(AU^{-1})$ have the same places of leading nonzero entries as $Q$ does. 
Moreover, every $u_{i,j}\ne 0$ for $i<j\le n$ implies that either $R_{j}(AU^{-1})$ is zero or the place of leading nonzero entry of $R_{j}(AU^{-1})$ is after that of $R_i(AU^{-1})$. 
Therefore, the leading nonzero entry of $R_i(\widetilde{A^\infty})$ equals that of $R_i(AU^{-1})$, namely $1$. We get $\widetilde{A^\infty}\in QU_n$. 

Finally, Theorem \ref{thm: ESO on coset} shows that $A$ can be  transformed via a finite number of ESOs stabilizing $Q U_n$
to   $\widetilde{A^\infty}$, and $\widetilde{A^\infty}$ is $D_n$-similar to $A^\infty$. 
\end{proof}

In summary, here is a simplification process to get \Bform\ $A^\infty$ of a given $A\in N_n$ under the $B_n$-similarity:
\begin{enumerate}
\item  
Use elementary row and column operations (cf. the proof of Lemma \ref{thm: double coset}) to factorize $A=B  Q B'$ for $B, B'\in B_n$
and $Q\in N_n$ is the subpermutation  determined by $ \{r^{i,j}(A): i,j\in [n]\}$. Then $A\SIM{B_n} QB'B$.

\item Write $B'B=DU$ for $D\in D_n$ and $U\in U_n$. Find $D'\in D_n$ such that $D'QD(D')^{-1}=Q$. Then 
$$QB'B=QDU\SIM{D_n} D' QDU (D')^{-1}=QD'U(D')^{-1}\in QU_n.$$

\item Use a sequence of ESOs stabilizing $QU_n$ to simplify $QD'U(D')^{-1}$ to a matrix $\widetilde{A^\infty}$ which is $D_n$-similar to   $A^\infty$ (cf. Theorem \ref{thm: ESO on coset} and Theorem \ref{thm: U_n similar to Bform}). Then determine $A^\infty$ (cf. Theorem \ref{thm: D_n similarity}). 

\end{enumerate}

We will explore the details of step (3) above in the coming sections.

\section{Graph representations and graph operations}\label{sect: graph operations}

In this section,  given a subpermutation   $Q\in N_n$, we use graph representations to visualize matrices in $QU_n$,
then use graph operations to visualize ESOs on matrices in $QU_n$.

\subsection{Graph representation of matrices in $QU_n$}\label{sect: graph representation}

Every $A=[a_{ij}]\in M_n$ is the adjacency matrix of a  directed graph 
$G_A=(V_A,E_A)$ with a weight function $w_A:[n]\times [n]\to \F\setminus\{0\}$ whose support is $E_A$. Precisely, 
\begin{equation}
V_A=[n];\qquad 
E_A=\{(i,j)\in[n]\times [n]: a_{ij}\ne 0\};\qquad
w_A(i,j)=a_{ij}. 
\end{equation}
Each element of $V_A$ (resp. $E_A$) is called a {\em vertex} (resp. an {\em arc}) of the graph $G_A$. 
Each arc $(i,j)\in E_A$  is visualized as  $i\ra j$, in which $i$ (resp. $j$) is called the {\em tail} (resp. the {\em  head}) of the arc $(i,j)$, and $w_A(i,j)$ is called {\em the weight} of the arc $(i,j)$. 
Call $G_A=(V_A,E_A)$ {\em the graph   of $A$}, and
$\widetilde{G}_A =(V_A,E_A,w_A)$ {\em the weighted graph   of $A$}, respectively. 

When $A\in N_n$, the graph of $A$ is simple and it consists of some arcs    $(i,j)\in[n]\times [n]$ with $i<j$.

A partition of $[n]$ has the form $[n]=S_1\cup\cdots\cup S_m$ where each partition subset $S_i\ne\emptyset$. For the uniqueness   of expression, we assume that
the minimal elements of $S_1,\ldots,S_m$ are in ascending order, and write the partition as $\widetilde S_1|\cdots| \widetilde S_m$ where $\widetilde S_i$ is the list of elements of $S_i$ in ascending order. For example, the parition $\{5,6\}\cup\{7,3\}\cup\{2,4,1\}$ of $[7]$ will be expressed as $124|37|56$ (for $n>9$, we will add spaces between neighboring numbers).

\begin{lem}\label{thm: graph partition subpermutation}
Given a subpermutation $Q\in N_n$,  the graph $G_Q$  of $Q$ consists of finite connected components, 
  each of which is a  directed path 
of the form:
\begin{equation}\label{chain}
i_1\lra i_2\lra\cdots\lra i_p,\qquad i_1<\cdots<i_p,\quad p\in\Z^+.
\end{equation}
There is a bijective correspondence between the set of all subpermutations in $N_n$ and the set of all paritions of $[n]$, in which 
$Q$ corresponds to the partition ${\mc P}_Q$ of the union of the sets $\{i_1,\ldots,i_p\}$, namely, the $(i,j)$ entry of $Q$ is nonzero if and only if
$i<j$ are sequential elements in a partition subset of ${\mc P}_Q$. 
\end{lem}

\begin{proof} Since $Q\in N_n$, the graph $G_Q$ only contains arcs  $(i,j)$ with $i<j$. Since $Q$ is a subpermutation, each row and column of $Q$ has at most one nonzero entry, so that each vertex $i$ of $G_Q$ is the head (resp. the tail) of at most one arc. Therefore, each connected component of $G_Q$ must have the form \eqref{chain}. The rest is obvious.
\end{proof}

We call each connected component subgraph \eqref{chain} of $G_Q$ a {\em chain of $G_Q$}. So the graph $G_Q$ is a union of finite disconnected chains. 
$G_Q$ is connected if and only if $Q$ is indecomposable.  When $Q$ is fixed, in the chain \eqref{chain}:
\begin{itemize}
\item $i_1$ (resp. $i_p$) is called {\em the chain tail} (resp. {\em the chain head}) of the chain \eqref{chain};
\item for each $k\in [p-1]$, $i_{k+1}$ is called {\em the chain successor} of $i_k$, denoted by
${i_k}^+=i_{k+1}$; and $i_{k}$ is called {\em the chain predecessor} of $i_{k+1}$, denoted by ${i_{k+1}}^-=i_k$.
\end{itemize}
We   call the partition ${\mc P}_Q$ in Lemma \ref{thm: graph partition subpermutation} {\em the partition of $Q$}. 
 ${\mc P}_Q$ also determines the permutation matrices $P$ in which each $PQP^T$ is a direct sum of indecomposable submatrices.


Lemma \ref{thm: subperm coset} for a subpermutation $Q\in N_n$ can be rephrased in graphs as follows. 

\begin{lem}\label{thm: QU_n graph}
$A\in M_n$ is in $QU_n$ for a subpermutation $Q\in N_n$ if and only if the weighted graph of $A$ satisfies the following conditions: 
\begin{enumerate}
\item $E_A\supseteq E_Q$ and the weights $w_A(i,j)=w_Q(i,j)=1$ for all $(i,j)\in E_Q$;
\item each $(i,j)\in E_A\setminus E_{Q}$ satisfies that $i$ is not the chain head of any chain of $G_Q$, and $i<i^+<j$ where $(i,i^+)\in E_Q$. 
\end{enumerate}
\end{lem}

\begin{proof} Suppose $A\in QU_n$ in which $Q\in N_n$ is  a subpermutation.
Write $Q=\sum_{i\in I} E_{i,\sigma(i)}$ for $I\subseteq [n]$. Then $i^+=\sigma(i)$ for all $i\in I$. Moreover, $i$ is a chain head of $G_Q$ if and only if $i\in [n]\setminus I$. Lemma \ref{thm: subperm coset} (2) shows that $\widetilde{G}_A$ contains $\widetilde{G}_Q$ as a weighted subgraph. 
Given $(i,j)\in E_A\setminus E_Q$,  Lemma \ref{thm: subperm coset} (1) shows that $i$ is not a chain head of $G_Q$, and 
Lemma \ref{thm: subperm coset} (2) and the assumption $Q\in N_n$  show that $i<\sigma(i)=i^+<j$. 

The converse statement also holds by Lemma \ref{thm: subperm coset}.
\end{proof}

In Lemma \ref{thm: QU_n graph}, $G_A$  contains $G_Q$ as a subgraph. 
When $A\in QU_n$ for  a subpermutation $Q$, we call each element of $E_A\setminus E_Q$  an {\em extra arc of $G_A$}. 
We denote the {\em graph type} of $A$ as ${\mc P}_Q: i_1j_1|\cdots|i_t j_t$ where ${\mc P}_Q$ is the partition corresponding to $Q$ and 
$(i_1,j_1),\ldots,(i_t,j_t)$ are the extra arcs of $G_A$ listed in ascending \Bel `s order \eqref{Border}. If $A$ has no extra arc (i.e. $A=Q$), 
its graph type is denoted as ${\mc P}_Q:\emptyset$. 
The graph type of $A$ is a concise expression of the graph $G_A$. 

\begin{ex}\label{ex: graph of QU_n}
Let $n=7$. Let 
$$Q=\smtx{0 &1&0&0&0&0&0\\&0&0&1&0&0&0\\&&0&0&0&0&1\\&&&0&0&0&0\\&&&&0&1&0\\&&&&&0&0\\&&&&&&0}\in N_7,\qquad
A=Q\smtx{1 &*  &*  &*  &*  &*  &*  \\ &1&0&3 &-2 &0 &1\\&&1&*  &*  &*  &*  \\&&&1&-1&0&0\\&&&&1&* &* \\&&&&&1&0\\&&&&&&1}=
\smtx{0 &1&0&3 &-2 &0 &1\\&0&0&1&-1&0&0\\&&0&0&0&0&1\\&&&0&0&0&0\\&&&&0&1&0\\&&&&&0&0\\&&&&&&0}.
$$
The subpermutation $Q$ corresponds to the partition ${\mc P}_Q=124|37|56$ of $[7]$. The graph of $Q$ is
$G_Q=([7], E_Q)$ in which $E_Q=\{(1,2),(2,4),(3,7),(5,6)\}$. The graph of $A$ is $G_A=([7],E_A)$ in which $E_A=E_Q\cup\{(1,4),(1,5),(1,7),(2,5)\}.$
So  the  graph type of $A$ is $124|37|56: 25|14|15|17$. 
\\ \\
\twopage{0.59}{0.42}
{
In the graph on the right, $\wt{G}_Q$  consists of three chains formed by black arcs with weights $1$, and $\wt{G}_A$ has the extra arcs with weights marked in red. By Lemma \ref{thm: QU_n graph}, an arc like $(4,7)$ or $(3,5)$ cannot be an extra arc of $G_A$. 
}
{
\begin{tikzpicture}
\node (1) at (0,2) {$1$};
\node (2) at (0.8,2) {$2$};
\node (3) at (1.6,1.3) {$3$};
\node (4) at (2.4,2) {$4$};
\node (5) at (3.2,0.6) {$5$};
\node (6) at (4,0.6) {$6$};
\node (7) at (4.8,1.3) {$7$};

	\foreach \from/\to in {1/2,2/4,3/7,5/6}
            \draw[edge] (\from) -- (\to);

	\foreach \from/\to in {2/5,1/7}
            \path[Redge] (\from) -- (\to);
	\node at (2.6,1.1) {\red{\tiny $-1$}}; 
	\node at (2.9,1.7) {\red{\tiny $1$}}; 

	\foreach \from/\to in {1/4}
            \draw[Redge] (\from) to [bend left = 20] (\to);
	\node at (1.2, 2.4) {\red{\tiny $3$}};  
                                                
	\foreach \from/\to in {1/5}
            \draw[Redge] (\from) to [bend right = 10] (\to);
	\node at (1.2,1.0) {\red{\tiny $-2$}};  

\end{tikzpicture}
}
\end{ex}

\subsection{The elementary $U_n$-similarity graph operations on $QU_n$}
The ESO in Definition \ref{def: ESO} can be rephrased using graph operations. 
Let $\lambda\in\F$ and $(p,q)\in [n]\times [n]$ with $p<q$.
For $A=[a_{ij}]\in N_n$, \eqref{elementary similarity} shows that
$$
\O_{p,q}^{\lambda}(A)=A+\sum_{\substack{j\in [n]\\ a_{qj}\ne 0}}\lambda a_{qj}E_{pj}-\sum_{\substack{i\in [n]\\a_{ip}\ne 0}}  \lambda a_{ip}E_{iq}.
$$
Let $A'=[a_{ij}']:=\O_{p,q}^{\lambda}(A)$. 
The changes made by $\O_{p,q}^{\lambda}$  from $\widetilde{G}_A$ to $\widetilde{G}_{A'}$ are below:
\begin{enumerate}
\item whenever $(i,p)\in E_A$ (i.e. $a_{ip}\ne 0$),  
\begin{equation}\notag
w_{A'}(i,q)=a_{iq}'=a_{iq}-\lambda a_{ip}=w_A(i,q)-\lambda w_A(i,p);
\end{equation}
\item whenever $(q,j)\in E_A$ (i.e. $a_{qj}\ne 0$), 
\begin{equation}\notag
w_{A'}(p,j)=a_{pj}'=a_{pj}+\lambda a_{qj}=w_A(p,j)+\lambda w_A(q,j).
\end{equation}
\end{enumerate}
These changes are visualized as follows, in which a red arc indicates a change of the weight, and a dashed arc indicates that the weight may be zero: 
\begin{equation}\label{ESGO}
\begin{array}{ccc}
\widetilde{G}_A & \overset{\O_{p,q}^{\lambda}}{\Longrightarrow}
&  \widetilde{G}_{A'}
\\
\begin{tikzpicture}
\node (1) at (0,2) {$i$};
\node (2) at (1.4,2) {$p$};
\node (3) at (2.1,1) {$q$};
\node (4) at (3.5,1) {$j$};
\node at (1.75,1.5)  {{\tiny $<$}};
\node at (0.7,2.3) {{\tiny $a_{ip}$}};
\node at (2.8,1) {{\tiny $a_{qj}$}};
\node at (1.05,1.2) {{\tiny $a_{iq}$}};
\node at (2.45,1.9)  {{\tiny $a_{pj}$}};
	\foreach \from/\to in {1/2}
           \draw [edge] (1) to [bend left = 20] (2);
	\foreach \from/\to in {3/4}
           \draw [edge] (3) to [bend right = 20] (4);
	\foreach \from/\to in {1/3}
           \draw [edge, dashed] (1) to [bend right = 10] (3);
	\foreach \from/\to in {2/4}
           \draw [edge, dashed] (2) to [bend left = 10] (4);
\end{tikzpicture}
& \overset{\O_{p,q}^{\lambda}}{\Longrightarrow}
&
\begin{tikzpicture}
\node (1) at (0,2) {$i$};
\node (2) at (1.4,2) {$p$};
\node (3) at (2.1,1) {$q$};
\node (4) at (3.5,1) {$j$};
\node at (1.75,1.5)  {{\tiny $<$}};
\node at (0.7,2.3) {{\tiny $a_{ip}$}};
\node at (2.8,1) {{\tiny $a_{qj}$}};
\node at (0.6,1.25) {\red{\tiny $a_{iq}-\lambda a_{ip}$}};
\node at (2.9,1.85)  {\red{\tiny $a_{pj}+\lambda a_{qj}$}};
	\foreach \from/\to in {1/2}
           \draw [edge] (1) to [bend left = 20] (2);
	\foreach \from/\to in {3/4}
           \draw [edge] (3) to [bend right = 20] (4);
	\foreach \from/\to in {1/3}
           \draw [edge, red, dashed] (1) to [bend right = 10] (3);
	\foreach \from/\to in {2/4}
           \draw [edge, red, dashed] (2) to [bend left = 10] (4);
\end{tikzpicture}
\end{array}
\end{equation}
By abuse of language, we also call transformation  \eqref{ESGO} the 
{\em elementary $U_n$-similarity operation (ESO)} $\O_{p,q}^{\lambda}$ on the weighted graph $\wt{G}_A$, denoted by $\O_{p,q}^{\lambda}(\wt{G}_A)=\wt{G}_{A'}.$

Given a matrix $A=[a_{ij}]\in QU_n$ where $Q\in N_n$ is a subpermutation,  Theorem \ref{thm: U_n similar to Bform} shows that
$A$ could be transformed via a sequence of ESOs   stabilizing $QU_n$     to a matrix $\wt{A^\infty}\in QU_n$ which is $D_n$-similar to   \Bform \ $A^\infty$.
By Theorem \ref{thm: D_n similarity}, $\wt{A^\infty}$ and $A^\infty$ have the same places of nonzero entries, that is, $G_{\wt{A^\infty}}=G_{A^\infty}$,
and the relation of weight functions \eqref{D_n similarity relation} holds on each undirected cycle of $G_{\wt{A^\infty}}$.
Therefore, we will use ESOs \eqref{ESGO} to eliminate ``redundant arcs''  on $\wt{G}_A$ following \Border\ until we reach $\wt{G}_{\wt{A^\infty}}$; then we adjust the weights  on undirected cycles of $\wt{G}_{\wt{A^\infty}}$ following \Border\ and get $\wt{G}_{A^\infty}$. 

\begin{lem}\label{thm: ESO necessary condition}
Let $A=[a_{ij}]\in N_n$. An arc $(i,j)$ of $\wt{G}_A=([n],E_A, w_A)$ can be eliminated by an ESO only if one of the following two cases happens:
\begin{enumerate}
\item
there is $p$ such that $i<p<j$ and $(i,p)\in E_A$, in which $\O_{p,j}^{\lambda}(\wt{G}_A)$  for $\lambda=\frac{a_{ij}}{a_{ip}}$ has no arc $(i,j)$; 

\item
there is $q$ such that $i<q<j$ and $(q,j)\in E_A$, in which $\O_{i,q}^{\lambda}(\wt{G}_A)$  for $\lambda=-\frac{a_{ij}}{a_{qj}}$ has no arc $(i,j)$; 

\end{enumerate}
\end{lem}

\begin{proof} The statement is a direct consequence of \eqref{ESGO}.
\end{proof}

By Theorem \ref{thm: U_n similar to Bform}, when $A\in QU_n$ for a subpermutation $Q\in N_n$, we should try to eliminate the extra arcs of $\wt{G}_A$
by ESOs stabilizing $QU_n$.  So in practice, not every ESO satisfying conditions in Lemma \ref{thm: ESO necessary condition} will be considered. 

\begin{ex}\label{ex: find Bform}
Let subpermutation $Q\in N_7$ and matrix $A\in QU_7$ be given in Example \ref{ex: graph of QU_n}. The extra arcs in $\wt{G}_A$ sorted by \Border\ are:
$(2,5)\prec (1,4)\prec (1,5)\prec (1,7)$.  We  test and  eliminate them by ESOs stabilizing $QU_7$ in this order. The first arc $(2,5)$ cannot be eliminated, since 
the only type of ESOs that can modify the weight of $(2,5)$ is $\O_{4,5}$ which creates the arc $(4,6)$ and does not stabilize $QU_7$. Then:
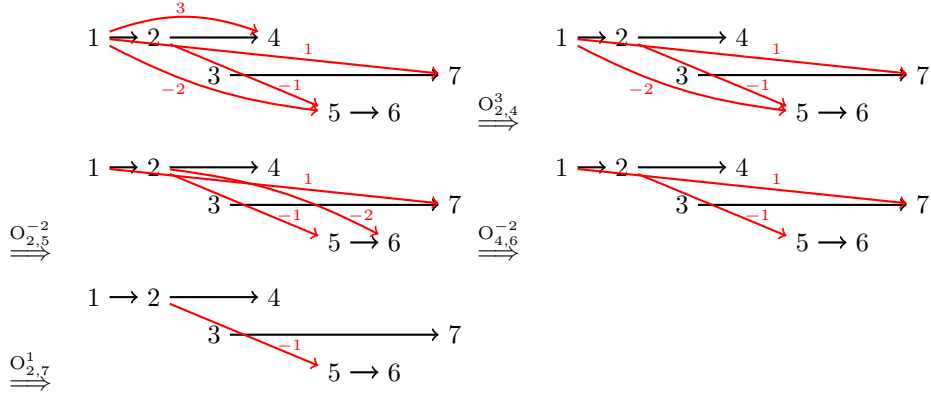
\begin{align*}
&&
\begin{tikzpicture}
\node (1) at (0,2) {$1$};
\node (2) at (0.8,2) {$2$};
\node (3) at (1.6,1.5) {$3$};
\node (4) at (2.4,2) {$4$};
\node (5) at (3.2,1) {$5$};
\node (6) at (4,1) {$6$};
\node (7) at (4.8,1.5) {$7$};
	\foreach \from/\to in {1/2,2/4,3/7,5/6}
            \draw[edge] (\from) -- (\to);
	\foreach \from/\to in {2/5,1/7}
            \path[Redge] (\from) -- (\to);
	\node at (2.65,1.35) {\red{\tiny $-1$}}; 
	\node at (2.9,1.85) {\red{\tiny $1$}}; 
	\foreach \from/\to in {1/4}
            \draw[Redge] (\from) to [bend left = 20] (\to);
	\node at (1.2, 2.4) {\red{\tiny $3$}};  
	\foreach \from/\to in {1/5}
            \draw[Redge] (\from) to [bend right = 10] (\to);
	\node at (1.1,1.3) {\red{\tiny $-2$}};  
\end{tikzpicture}
&\overset{\O_{2,4}^{3}}{\Lra}&
\begin{tikzpicture}
\node (1) at (0,2) {$1$};
\node (2) at (0.8,2) {$2$};
\node (3) at (1.6,1.5) {$3$};
\node (4) at (2.4,2) {$4$};
\node (5) at (3.2,1) {$5$};
\node (6) at (4,1) {$6$};
\node (7) at (4.8,1.5) {$7$};
	\foreach \from/\to in {1/2,2/4,3/7,5/6}
            \draw[edge] (\from) -- (\to);
	\foreach \from/\to in {2/5,1/7}
            \path[Redge] (\from) -- (\to);
	\node at (2.65,1.35) {\red{\tiny $-1$}}; 
	\node at (2.9,1.85) {\red{\tiny $1$}}; 
	\foreach \from/\to in {1/5}
            \draw[Redge] (\from) to [bend right = 10] (\to);
	\node at (1.1,1.3) {\red{\tiny $-2$}};  
\end{tikzpicture}
\\
&\overset{\O_{2,5}^{-2}}{\Lra}&
\begin{tikzpicture}
\node (1) at (0,2) {$1$};
\node (2) at (0.8,2) {$2$};
\node (3) at (1.6,1.5) {$3$};
\node (4) at (2.4,2) {$4$};
\node (5) at (3.2,1) {$5$};
\node (6) at (4,1) {$6$};
\node (7) at (4.8,1.5) {$7$};
	\foreach \from/\to in {1/2,2/4,3/7,5/6}
            \draw[edge] (\from) -- (\to);
	\foreach \from/\to in {2/5,1/7}
            \path[Redge] (\from) -- (\to);
	\node at (2.65,1.35) {\red{\tiny $-1$}}; 
	\node at (2.9,1.85) {\red{\tiny $1$}}; 
	\draw[Redge] (2) to [bend left = 10] (6); 
	\node at (3.6,1.35) {\red{\tiny $-2$}};  
\end{tikzpicture}
&\overset{\O_{4,6}^{-2}}{\Lra}&
\begin{tikzpicture}
\node (1) at (0,2) {$1$};
\node (2) at (0.8,2) {$2$};
\node (3) at (1.6,1.5) {$3$};
\node (4) at (2.4,2) {$4$};
\node (5) at (3.2,1) {$5$};
\node (6) at (4,1) {$6$};
\node (7) at (4.8,1.5) {$7$};
	\foreach \from/\to in {1/2,2/4,3/7,5/6}
            \draw[edge] (\from) -- (\to);
	\foreach \from/\to in {2/5,1/7}
            \path[Redge] (\from) -- (\to);
	\node at (2.65,1.35) {\red{\tiny $-1$}}; 
	\node at (2.9,1.85) {\red{\tiny $1$}}; 
\end{tikzpicture}
\\
&\overset{\O_{2,7}^{1}}{\Lra}&
\begin{tikzpicture}
\node (1) at (0,2) {$1$};
\node (2) at (0.8,2) {$2$};
\node (3) at (1.6,1.5) {$3$};
\node (4) at (2.4,2) {$4$};
\node (5) at (3.2,1) {$5$};
\node (6) at (4,1) {$6$};
\node (7) at (4.8,1.5) {$7$};
	\foreach \from/\to in {1/2,2/4,3/7,5/6}
            \draw[edge] (\from) -- (\to);
	\foreach \from/\to in {2/5}
            \path[Redge] (\from) -- (\to);
	\node at (2.65,1.35) {\red{\tiny $-1$}}; 
\end{tikzpicture}
\end{align*}
Hence $\O_{2,7}^{1}\O_{4,6}^{-2}\O_{2,5}^{-2}\O_{2,4}^{3}(A)=\wt{A^\infty}$ in which $\wt{G}_{\wt{A^\infty}}$ is the last weight graph above. 
There is no undirected cycle on $\wt{G}_{\wt{A^\infty}}$.  So $\wt{A^\infty}$ is $D_7$-similar to \Bform\ $A^\infty$ whose weighted graph has weight $1$ on the arc $(2,5)$. In other words,
$$A\SIM{U_7}\wt{A^\infty}=Q-E_{2,5} \SIM{D_7}A^\infty =Q+E_{2,5}.
$$
\end{ex}

The elimination process in Example \ref{ex: find Bform} may be roughly abbreviated as changes on graph types as below, where an appropriate operation on
the first row causes the changes of arcs listed on the second row.
\begin{equation}\label{to get Bform}
\begin{matrix}
&\O_{2,4}
&\O_{2,5}
&\O_{4,6}
&\O_{2,7}
&
\\
124|37|56: 
25|14|15|17 
&-14
&-15+26
&-26
&-17 
&= 25
\end{matrix}
\end{equation}
So \Bform \ of $A$ has the type $124|37|56: 25$. A process like \eqref{to get Bform}  only works for  generic cases, since the process omits all weight information; for some specific matrix, an ESO on its weighted graph may eliminate several extra arcs simultaneously and lead to a different type. 
However,  we will see that a graphical version of process \eqref{to get Bform} is powerful in classifying the forms of \Bform s under $B_n$-similarity.

Another observation about Example \ref{ex: find Bform} is that: unlike those ESOs in Theorem \ref{thm: ESO on coset}, in $\O_{2,7}^{1}\O_{4,6}^{-2}\O_{2,5}^{-2}\O_{2,4}^{3}(A)=\wt{A^\infty}$,
the pairs $(2,4), (2,5), (4,6), (2,7)$ do not completely follow \Border\ $\prec$. 
However, we check and (if possible) eliminate the extra arcs of $\wt{G}_A$ following \Border; the success of this process is guaranteed by the combination of \Balg \ and Theorem \ref{thm: ESO on coset}.

\section{Properties of \Bform s\ under $B_n$-similarity}\label{sect: Bform properties}

Given a matrix $A\in B_n QB_n$ or $QU_n$ where  $Q\in N_n$ is a subpermutation, Theorem \ref{thm: Bform in right coset} shows that \Bform\ $A^\infty\in QU_n$. 
Here we investigate the nonzero entries in $A^\infty$, or equivalently, what extra arcs and weights could be in $\wt{G}_{A^\infty}$. 
For simplicity, we assume that $A$ is already a \bform.

\subsection{Characterization of \Bform}

The  \eqref{D_n similarity relation} in Theorem \ref{thm: D_n similarity} indicates that
if the graph of a \bform\ has an undirected cycle, then at least one arc of this undirected cycle has a parameter weight. 
It derives the following results.

\begin{thm}\label{thm: Bform number of parameters}
Let $A\in N_n$ be a  \bform. 
If $G_A$ has $m$ connected components and $|E_A|=N$,
then $A$ has $m$ indecomposable components and $N-n+m$ parameters. 
\end{thm}

\begin{proof}
If $G_A$ has $m$ connected components, then the vertex sets of these $m$ connected subgraphs form a partition of $[n]$. 
For each permutation matrix $P\in M_n$, the graphs $G_{PAP^T}$ and $G_A$ are isomorphic. 
There is a permutation matrix $P$ such that the vertex set of each   connected component of $G_{PAP^T}$ contains sequential integer(s). 
Then $PAP^T$ is a direct sum of $m$ principal submatrices, each of which is indecomposable. 
In other words, $A$ has $m$ indecomposable components.

If a connected component of $G_A$ has $n_1$ vertices and $r_1$ arcs, then  $r_1\ge n_1-1$.
When $r_1=n_1-1$, the connected component contains no undirected cycle so that all weights of its arcs are 1 by Theorem \ref{thm: D_n similarity} and
\Balg. 
When $r_1>n_1-1$, the connected component can be obtained by adding $r_1-n_1+1$ arcs to a connected subgraph with $n_1-1$ arcs, and adding each arc creates  an undirected cycle on the union of this arc and the subgraph.
Therefore, by Theorem \ref{thm: D_n similarity}, there are $r_1-n_1+1$ parameter weights on the arcs of this connected component.

Summing over all $m$ connected components of $G_A$,
we see that $A$ has $N-n+m$ parameters.
\end{proof}

\begin{rem}
In matrix way, Theorem \ref{thm: Bform number of parameters} says that:
if a \bform\ $A\in N_n$ is permutation similar to a direct sum of $m$ indecomposable squared submatrices, and $A$ has $N$ nonzero entries, then $A$ has $N-n+m$ parameters.
\end{rem}

The following two results describe an indecomposable \bform\ and its graph. They show that if the graph type or the places of nonzero entries of a \bform\ are known, then 
we can determine the amount and the places of parameters among these nonzero entries. 

\begin{cor}\label{thm: indecomposable Bform parameters}
Let $A\in N_n$ be a \bform. Then $A$ is indecomposable if and only if the graph $G_A=([n],E_A)$ is connected.  
Moreover, if $A$ is indecomposable with $N$ nonzero entries, then $A$ has $N-n+1$ parameters. 
\end{cor}

\begin{thm}\label{thm: Bform places of parameters}
Let $A\in QU_n$ be an indecomposable \bform\ in which $Q\in N_n$ is a subpermutation. 
List  the extra arcs of $G_A$ (i.e. the elements of $E_A\setminus E_Q$) in \Border:
\begin{equation}\label{extra arc Border}
(i_1,j_1)\prec (i_2,j_2) \prec\cdots\prec (i_t,j_t).
\end{equation}
Then the places of parameters of $A$ (if any) correspond to the marked extra arcs  determined by the following steps, 
starting at the graph $G:=G_Q$ in which all arcs in $E_Q$ are unmarked:
\begin{enumerate}
\item
add  the extra arcs of $G_A$ one at a time to $G$ according to \Border\ \eqref{extra arc Border}. 
\item
when adding an extra arc $(i,j)$ to $G$ creates an undirected cycle  in which none of the arcs is marked, 
mark the extra arc $(i,j)$  and continue;
 
\item 
repeat the steps (1) and (2) until all extra arcs of $G_A$ are gone through. 
\end{enumerate}
\end{thm}

\begin{proof} 
Since $A\in QU_n$, the parameters   of $A$ appear only in the entries corresponding to extra arcs.  

In step (2), 
when adding an extra arc $(i,j)$  results in an undirected cycle  in which none of the arcs is marked, 
we may assume that the undirected cycle  has distinct vertices  by removing redundant subcycles. 
By Theorem \ref{thm: D_n similarity} (2), the undirected cycle contains at least one arc with a parameter weight to represent the scalar in \eqref{D_n similarity relation}. 
Moreover, $(i,j)$ is the last arc in \Border\ in this undirected cycle. So by \Balg, the parameter weight in the undirected cycle should be on $(i,j)$.

 After step (3), 
if we remove all marked arcs from $G_A$ then the remaining subgraph does not have any undirected cycle. By Theorem \ref{thm: D_n similarity},
$A$ is $D_n$-similar (and thus $B_n$-similar) to a matrix whose unmarked arcs have weights 1 and marked arcs have parameter weights. 

The normalization steps (1), (2), (3) 
allow us to place the parameters of $A$ in accordance with \Balg. So these steps determine the places of parameters.
\end{proof}

\begin{ex} An analysis similar to Example \ref{ex: find Bform} 
shows that:  
every matrix $A\in QU_8$ of 
the graph type $123678|45: 46|24|14$ 
has no extra arc in $G_A$ that can be eliminated by ESOs stablizing $QU_8$.  
So $A$ is a \bform, which is indecomposable since  $G_A$ is connected. Corollary \ref{thm: indecomposable Bform parameters} shows that $A$ has $2$ parameters, and 
Theorem \ref{thm: Bform places of parameters} shows that the parameters appear in the $(2,4)$ and $(1,4)$ entries. So  $\wt{G}_A$ and $A$ have the forms ($\lambda,\mu\in\F\setminus\{0\}$): 
\\
\twopage{.55}{.45}
{
\begin{center}
\begin{tikzpicture}
\node at (-0.7,0) {$\wt{G}_A:$};
\node (1) at (0,0) {$1$};
\node (2) at (0.8,0) {$2$};
\node (3) at (1.6,0) {$3$};
\node (4) at (2.4,-0.5) {$4$};
\node (5) at (3.2,-0.5) {$5$};
\node (6) at (4,0) {$6$};
\node (7) at (4.8,0) {$7$};
\node (8) at (5.6,0) {$8$};

	\foreach \from/\to in {1/2,2/3,3/6,6/7,7/8,4/5}
            \draw[edge] (\from) -- (\to);
	\foreach \from/\to in {4/6,2/4,1/4}
            \path[Redge] (\from) -- (\to);
	\node at (3,-0.2) {\red{\tiny $1$}}; 
	\node at (1.8,-0.2) {\red{\tiny $\lambda$}}; 
	\node at (1.1,-0.35) {\red{\tiny $\mu$}}; 
\end{tikzpicture}
\end{center}
}
{
$$A=\smtx{0&1&0&\mu&0&0&0&0\\&0&1&\lambda&0&0&0&0\\&&0&0&0&1&0&0\\&&&0&1&1&0&0\\&&&&0&0&0&0\\&&&&&0&1&0\\&&&&&&0&1\\&&&&&&&0} 
$$}
All such \bform s may be represented by the graph type with additional underlines indicating parameters, namely $123678|45: 46|\ul{24}|\ul{14}$. 
\end{ex}

\subsection{Extra arcs in the graph of \Bform}

Fix a subpermutation 
$Q=\sum_{i\in I} E_{i,\sigma(i)}\in N_n. $
Using the notations in Section \ref{sect: graph representation}, 
the graph $G_Q$ consists of $n-|I|$ chains in which the set $S_{h}$ of chain heads  and the set $S_{t}$ of  chain tails are 
\begin{equation}
S_{h}=[n]\setminus I,\qquad S_{t}=[n]\setminus \sigma(I). 
\end{equation}
We also denote the maps 
\begin{equation}
I\to\sigma(I),\quad i\mapsto i^+:=\sigma(i),\qquad \text{and \ } 
\sigma(I)\to I,\quad j\mapsto j^-:=\sigma^{-1}(j). 
\end{equation}

Given a \bform\ $A\in QU_n$, Lemma \ref{thm: subperm coset} and its graph version Lemma \ref{thm: QU_n graph} give  a description of the entries of $A$. 
In this subsection, we further explore what entries of $A$ should be zero, namely, 
what extra arcs should not be in $G_A$. 

\begin{thm}\label{thm: Bform no arcs}
Let $A=[a_{ij}]\in QU_n$ be a \bform\  in which $Q=\sum_{i\in I} E_{i,\sigma(i)}\in N_n$ is  a subpermutation.
Then for $i\in I$, $(i,j)\not\in E_A$ (i.e. $a_{ij}=0$) when one of the following situations happen:

\begin{enumerate}
\twopage{.5}{.5}{\item $i^+<j\in S_{h}$;\\
\begin{tikzpicture}
\node (1) at (0,1) {$i$};
\node (2) at (1.5,1) {$i^+$};
\node (4) at (2.5,0) {$j$};
	\foreach \from/\to in {1/2}
            \path[edge] (\from) -- (\to);
	\foreach \from/\to in {1/4}
            \path[Redge,dashed] (\from) -- (\to);
\node at (2,0.5)  {\tiny $<$};
\node at (2, 1)  {$\cdots$};
\node at (-0.4, 1)  {$\cdots$};
\node at (1.6,0)  {$\cdots$};
\node at (3,0)  {\tiny $\in S_{h}$};
\end{tikzpicture}
}
{\item  $j\not\in S_{t}$  and $i<j^-$.
\\ 
\begin{tikzpicture}
\node (1) at (0,1) {$i$};
\node (2) at (1.5,1) {$i^+$};
\node (3) at (1,0) {$j^-$};
\node (4) at (2.5,0) {$j$};
	\foreach \from/\to in {1/2,3/4}
            \path[edge] (\from) -- (\to);
	\foreach \from/\to in {1/4}
            \path[Redge, dashed] (\from) -- (\to);
\node at (0.5,0.5)  {\tiny $<$};
\node at (2, 1)  {$\cdots$};
\node at (-0.4, 1)  {$\cdots$};
\node at (0.5,0)  {$\cdots$};
\node at (3,0)  {$\cdots$};
\end{tikzpicture}
}
\end{enumerate}
In particular, if $i$ and $j$ are on the same chain of $G_Q$ but $j\ne i^+$, then $(i,j)\not\in E_A$. 
\end{thm}

\begin{proof}
We prove by contradictions that $A$ cannot be a \bform\ if an arc $(i,j)\in E_A$ satisfies (1) or (2) of Theorem \ref{thm: Bform no arcs}.
The idea is to find a matrix $A'\in Q U_n$ such that $A'\SIM{U_n} A$  and 
\begin{equation}\label{set removing an arc} 
E_{A'}\subseteq (E_{A}\setminus\{(i,j)\})\cup \{(i',j')\in [n]\times [n]: i'<j',\ (i,j)\prec (i',j')\},
\end{equation}
which contradicts \Balg\ to get \Bform\ $A$.
 
\begin{enumerate}
\item Suppose $(i,j)\in E_A$ such that $i^+<j\in S_{h}$. 
By Lemma \ref{thm: ESO necessary condition}, there is $\lambda_1\in\F$ such that the graph of $A_1:=\O_{i^+,j}^{\lambda_1}(A)$ 
contains no arc $(i,j)$. 
By Lemma \ref{thm: QU_n graph} (2), $j\in S_{h}$ is not the tail of any arc of $\wt{G}_A$. 
Thus by \eqref{ESGO}, $E_{A_1}\setminus E_A$ only contains some $(i_1,j)$ in which $(i_1,i^+)\in E_A$ and $i_1\ne i$ so that $i_1\in I$ and $i_1^+<i^+$ by Lemma \ref{thm: QU_n graph} (2). 
\\
\twopage{.5}{.4}{
\quad The changes from $G_A$ to $G_{A_1}$   are illustrated on the right, in which the dashed blue arc is removed and some solid blue arcs are added.
}
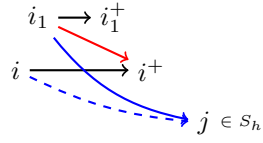
{
\begin{center}
\begin{tikzpicture}
\node (1) at (0,1) {$i$};
\node (2) at (1.8,1) {$i^+$};
\node (4) at (2.5,0.3) {$j$};
\node (5) at (0.3,1.7) {$i_1$};
\node (6) at (1.3,1.7) {$i_1^+$};
	\foreach \from/\to in {1/2,5/6}
            \path[edge] (\from) -- (\to);
	\foreach \from/\to in {1/4}
            \draw [Bedge, dashed] (\from) to [bend right = 10] (\to);
	\foreach \from/\to in {5/2}
            \path[Redge] (\from) -- (\to);
	\foreach \from/\to in {5/4}
            \draw [Bedge] (\from) to [bend right = 20] (\to);
\node at (3,0.3)  {\tiny $\in S_{h}$};
\end{tikzpicture}
\end{center}
} 

Similarly, for each $(i_1,j)\in E_{A_1}\setminus E_A$, an appropriate $\O_{i_1^+,j}$-operation will remove the arc $(i_1,j)$ from the graph and add the arcs $(i_2,j)$ in which $(i_2,i_1^+)\in E_{A_1}$ and $i_2\ne i_1$
so that $i_2\in I$ and $i_2^+<i_1^+$. Repeating the process results in $i^+>i_1^+>i_2^+>\cdots$. However, the process cannot go on forever. 
Hence by a finite steps of ESOs we can remove $(i,j)$ from $G_A$ as well as all arcs created by these ESOs. In other words,
we get $A'=A-a_{ij}E_{ij}$ such that $A'\SIM{U_n} A$. This contradicts  the assumption that $A$ is a \bform. 
Therefore, $(i,j)\not\in E_A$.

\item Suppose $(i,j)\in E_A$ such that $j\not\in S_{t}$  and $i<j^-$. 
By Lemma \ref{thm: ESO necessary condition}, there is $\lambda\in\F$ such that the graph of $A':=\O_{i,j^-}^{\lambda}(A)$ 
contains no arc $(i,j)$. By \eqref{ESGO},  $E_{A'}\setminus E_{A}$ only contains the following possible arcs:
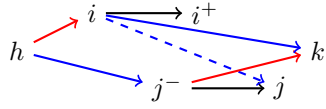
\begin{center}
\begin{tikzpicture}
\node (1) at (0,1) {$i$};
\node (2) at (1.5,1) {$i^+$};
\node (3) at (1,0) {$j^-$};
\node (4) at (2.5,0) {$j$};
\node (5) at (-1,0.5) {$h$};
\node (6) at (3,0.5) {$k$};
	\foreach \from/\to in {1/2,3/4}
            \path[edge] (\from) -- (\to);
	\foreach \from/\to in {1/4}
            \path[Bedge, dashed] (\from) -- (\to);
	\foreach \from/\to in {5/1, 3/6}
            \path[Redge] (\from) -- (\to);
	\foreach \from/\to in {5/3, 1/6}
            \path[Bedge] (\from) -- (\to);
\end{tikzpicture}
\end{center}
\begin{enumerate}
\item $(h,j^-)\in E_{A'}\setminus E_A$ in which $(h,i)\in E_A$. In such a case, $h<i$ so that $(h,j^-)\succ (i,j)$ in \Border. 

\item $(i,k)\in E_{A'}\setminus E_A$ in which $(j^-,k)\in E_A$ and $k\ne j$. In such a case, $j^-<j<k$ by Lemma \ref{thm: QU_n graph} (2) so that $(i,k)\succ (i,j)$ in \Border.
\end{enumerate}

Overall, we get  $A'\SIM{U_n} A$ in which $E_{A'}$ satisfies \eqref{set removing an arc}. It contradicts  the assumption that $A$ is a \bform. 
Therefore, $(i,j)\not\in E_A$. \qedhere
\end{enumerate}
\end{proof}

The   less intuitive matrix version of Theorem \ref{thm: Bform no arcs} is as follows.

\begin{thm}
Let  $A=[a_{ij}]\in QU_n$ be a \bform\ where $Q=\sum_{i\in I} E_{i,\sigma(i)}\in N_n$ is a subpermutation.
Then $a_{ij}=0$ whenever: 
\begin{enumerate}
\item $j\in [n]\setminus (I\cup\{\sigma(i)\})$, or 
\item $j\in\sigma(I)$ and $i<\sigma^{-1}(j)$. 
\end{enumerate} 
In particular, $a_{ij}=0$ if $j=\sigma^{k}(i)$ for some integer $k>1$.
\end{thm}

Theorem \ref{thm: Bform no arcs} is equivalent to the following result  which gives a characterization of the possible arcs in a \bform.

\begin{thm}\label{thm: Bform arcs}
Let $A=[a_{ij}]\in QU_n$ be a \bform\  in which $Q=\sum_{i\in I} E_{i,\sigma(i)}\in N_n$ is  a subpermutation. 
Then $(i,j)\in E_A$ (i.e. $a_{ij}\ne 0$) implies  $i\in I=[n]\setminus S_{h}$  and one of the following:
\begin{enumerate}
\item 
\twopage{.45}{.5}
{$j=i^+$ (where $a_{ij}=1$).}
{\begin{tikzpicture}
\node (1) at (1,0) {$i$};
\node (2) at (2.5,0) {$i^+=j$};
		\foreach \from/\to in {1/2}
	\path[edge] (\from) -- (\to);
\end{tikzpicture}
}
\item 
\twopage{.45}{.5}
{$j\in S_{t}\setminus S_{h}$ and $i^+<j$.}
{\begin{tikzpicture}
\node (1) at (0,0.7) {$i$};
\node (2) at (1.5,0.7) {$i^+$};
\node (3) at (2.5,0) {$j$};
\node (4) at (4,0) {$j^+$};
	\foreach \from/\to in {1/2,3/4}
            \path[edge] (\from) -- (\to);
	\foreach \from/\to in {1/3}
            \path[Redge] (\from) -- (\to);
\node at (2.5,-0.3)  {\tiny $\in S_{t}\setminus S_{h}$};
\node at (2,0.35)  {\tiny $<$};
\end{tikzpicture}
}
\item
\twopage{.45}{.5}
{$j\not\in S_{t}\cup S_{h}$ and $j^-<i<i^+<j$.}
{\begin{tikzpicture}
\node (1) at (0,0.7) {$i$};
\node (2) at (1.5,0.7) {$i^+$};
\node (3) at (-1,0) {$j^-$};
\node (4) at (2.5,0) {$j$};
	\foreach \from/\to in {1/2,3/4}
            \path[edge] (\from) -- (\to);
	\foreach \from/\to in {1/4}
            \path[Redge] (\from) -- (\to);
\node at (-0.5,0.35)  {\tiny $<$};
\node at (2,0.35)  {\tiny $<$};
\node at (3.2,0) {\tiny $\not\in S_{t}\cup S_{h}$};
\end{tikzpicture}
}
\end{enumerate}
In particular, given $i\in I$, there is at most one vertex $j$ in each chain of $G_Q$ such that
$(i,j)\in E_A$. 
\end{thm}

\begin{proof}
The case (1) is $(i,j)\in E_Q$. 
The cases (2) and (3) cover those extra arcs $(i,j)$  not included
in Theorem \ref{thm: Bform no arcs} (1) and (2). 

It remains to prove the last claim. 
Suppose $(i,j)\in E_A$ and the vertex $j$ is in a chain $G'$ of $G_Q$. 
If the chain $G'$ contains the vertex $i$, then $j=i^+$.
Otherwise, $j$ is the lowest vertex number in the chain $G'$ such that $j>i^+$.
\end{proof}

\begin{thm}\label{thm: long chain deletion}
Let $A=[a_{ij}]\in QU_n$ be a \bform\ in which $Q=\sum_{i\in I} E_{i,\sigma(i)}\in N_n$ is  a subpermutation. 
Given $(i,j)\in [n]\times [n]$ with $i<j$,  
suppose there exist $m\in\N$ and sequences $i_0=i, i_1,\ldots,i_m\in [n]$ and $j_0=j, j_1,\ldots,j_m\in [n]$ such that all of the following conditions hold for $p\in[m]$: 
\begin{center}
\begin{tikzpicture}
\node (1) at (0,1) {$i_0$};  
	\node at (-0.5, 1) {$i=$};  
\node (2) at (1.5,1) {$i_1$};
\node (3) at (3,1) {$i_2$};
\node (4) at (5,1) {$i_{m-1}$}; 
\node (5) at (6.5,1) {$i_m$};
\node (6) at (2,0.3) {$j_0$};
	\node at (1.5,0.3) {$j=$};  
\node (7) at (3.5,0.3) {$j_1$};
\node (8) at (5,0.3) {$j_2$};
\node (9) at (7,0.3) {$j_{m-1}$};
\node (10) at (8.5,0.3) {$j_m$};
	\foreach \from/\to in {1/2,2/3,4/5, 6/7,7/8,9/10}
            \path[edge] (\from) -- (\to);
	\foreach \from/\to in {1/6, 2/7, 3/8, 4/9,5/10}
            \path[edge, dashed] (\from) -- (\to);
\node at (4,1)  {\bf $\cdots\cdots$};
\node at (6,0.3)  {$\cdots\cdots$};
\end{tikzpicture}
\end{center}
\begin{enumerate}
\item 
$(i_{p-1},i_{p})$ is the only arc in $G_A$ whose head is $i_{p}$.  
\item 
$(j_{p-1},j_{p})=(j_{p-1},j_{p-1}^+)$ is the only arc in $G_A$ whose tail is $j_{p-1}$. 
\item $i_p<j_{p-1}$. 
\item $i_m\not\in S_{h}$ but $j_m\in S_{h}$. 
\end{enumerate}
Then $(i,j)\not\in E_A$. 
\end{thm}

\begin{proof} 
Suppose on the contrary $(i,j)=(i_0,j_0)\in E_A$. There exists $\lambda_1\in\F$ such that the graph of $A_1:=\O_{i_1,j_0}^{\lambda_1}(A)$ does not contain the arc $(i,j)$.
Then either $E_{A_1}=E_{A}\setminus \{(i,j)\}$ or $E_{A_1}=(E_{A}\setminus \{(i,j)\})\cup\{(i_1,j_1)\}$. 
However, $E_{A_1}=E_{A}\setminus \{(i,j)\}$  is impossible since $A$ is a \bform. 
So  $E_{A_1}=(E_{A}\setminus \{(i,j)\})\cup\{(i_1,j_1)\}$.  
Similarly, applying a sequence of appropriate $ \O_{i_2,j_1}, \ldots,\O_{i_m,j_{m-1}}$ operations to $A_1$,
we will get $A_m\SIM{U_n}A_1\SIM{U_n}A$ such that 
$$E_{A_m}=(E_{A}\setminus \{(i,j),(i_1,j_1),\ldots,(i_{m-1},j_{m-1})\})\cup\{(i_m,j_m)\}.$$ 
Since $i_m\not\in S_{h}$ and $j_m\in S_{h}$, 
the proof of Theorem \ref{thm: Bform no arcs} indicates that there is $A_{m+1}\SIM{U_n} A_m\SIM{U_n} A$ such that
$$E_{A_{m+1}}=E_{A_m}\setminus \{(i_m,j_m)\}=E_{A}\setminus \{(i,j),(i_1,j_1),\ldots,(i_{m},j_{m})\}.$$
It  contradicts the assumption that $A$ is a \bform.  
\end{proof}

\begin{ex} Let $Q\in N_8$ be the subpermutation with ${\mc P}_Q=12368|457$.
Consider the possible \bform s  $A\in QU_8$.  
 Theorem \ref{thm: Bform arcs} implies that the possible extra arcs in $E_A\setminus E_Q$ are $(1,4)$, $(2,4)$, and $(4,6)$. 
By Theorem \ref{thm: long chain deletion}, neither $(4,6)$ nor $(1,4)$ can be in a \bform. Therefore, the only indecomposable
\bform\ in $QU_8$ is of the graph form $12368|457:24$. \\
\twopage{.4}{.6}
{$$A=\smtx{0&1&0&0&0&0&0&0\\&0&1&1&0&0&0&0\\&&0&0&0&1&0&0\\&&&0&1&0&0&0\\&&&&0&0&1&0\\&&&&&0&0&1\\&&&&&&0&0\\&&&&&&&0
}$$
}
{
\begin{tikzpicture}
\node (1) at (0,0.8) {$1$};
\node (2) at (0.8,0.8) {$2$};
\node (3) at (1.6,0.8) {$3$};
\node (4) at (2.4,0) {$4$};
\node (5) at (3.2,0) {$5$};
\node (6) at (4,0.8) {$6$};
\node (7) at (4.8,0) {$7$};
\node (8) at (5.6,0.8) {$8$};
	\foreach \from/\to in {1/2,2/3,3/6,6/8,4/5,5/7}
            \draw[edge] (\from) -- (\to);
	\foreach \from/\to in {2/4}
            \path[Redge] (\from) -- (\to);
\end{tikzpicture}
}
\end{ex}

\subsection{Possible numbers of parameters in a \bform}

In \cite[Theorem 2.4]{Chen2016}, Chen et al showed that for $n\ge 6$, there exists an indecomposable \bform\ in $N_n$ which admits
at least $\lfloor\frac{n}{2}\rfloor-2$ parameters. Note that a matrix in $N_n$ has up to $\frac{n(n-1)}{2}$   nonzero entries. 
We show below the existence of indecomposable $3$-nilpotent \bform s with arbitrary number of parameters up to $\O(n^2)$. 

\begin{theorem}\label{thm: Bform parameters}
Let $n, r\in\N$ such that $n\ge 6$ and
\beq\label{Bform parameters}
r\le \begin{cases}
\frac{1}{2} \lfloor\frac{n-2}{3}\rfloor(\lfloor\frac{n-2}{3}\rfloor-1) &\text{if \ } n\equiv 0, 2 \mod 3,
\\
\frac{1}{2} \lfloor\frac{n-2}{3}\rfloor(\lfloor\frac{n-2}{3}\rfloor-1)-1 &\text{if \ } n\equiv 1 \mod 3.
\end{cases}
\eeq
Then there exists an  indecomposable  \bform \ $A\in N_n$ with $r$ parameters, and $A$ has the minimal polynomial $x^3$.  
\end{theorem}

\begin{proof}
We construct the desired \bform s for $n\ge 6$ according to $n\mod 3$:
\begin{enumerate}
\item
When $n=3m$,  choose the subpermutation $Q\in N_{3m}$ with
\beq
{\mc P}_Q=1\  2m\  2m+1\ |\ 2\ 2m-1\ 2m+2\ |\ \cdots\ |\ m\ m+1\ 3m.
\eeq
Let  $G=([3m],E)$ be the graph containing  $G_Q$ as a subgraph and
\beq
E\setminus E_Q=\{(i,j)\in [3m]\times[3m]: 2\le i\le m,\ 2m+2-i\le j\le 2m\}. 
\eeq
The graph $G$ is illustrated as below ($1\le i_1<i_2<i_3<\ldots\le n$):
\begin{center}
\begin{tikzpicture}
\node (1) at (0,2.1) {$i_1$};
\node (2) at (0.5,1.4) {$i_2$};
\node (3) at (1,0.7) {$i_3$};
\node (4) at (1.5,0) {$i_4$};
\node (5) at (3,0) {$i_5$};
\node (6) at (3.5,0.7) {$i_6$};
\node (7) at (4,1.4) {$i_7$};
\node (8) at (4.5,2.1) {$i_{8}$};
\node (9) at (6,2.1) {$i_{9}$};
\node (10) at (6.5,1.4) {$i_{10}$};
\node (11) at (7,0.7) {$i_{11}$};
\node (12) at (7.5,0) {$i_{12}$};
	\foreach \from/\to in {1/8,8/9,2/7,7/10,3/6,6/11,4/5,5/12}
            \path[edge] (\from) -- (\to);
	\foreach \from/\to in {2/8,3/8,3/7,4/8,4/7,4/6}
            \path[Redge] (\from) -- (\to);
\end{tikzpicture}
\end{center}
Let $A\in QU_n$ such that $G_A=G$ and the places of parameters in $A$ are given by Theorem \ref{thm: Bform places of parameters}. 
We claim that  
$A$ is a \bform. 

By \eqref{index set of U_n stabilizing QU_n}, the ESOs stabilizing $QU_n$ are those
$\O_{p,q}^{\lambda}$ in which   either 
\begin{enumerate}
\item $(p,q)\in  [m]\times ([2m]\setminus[m])$, or
\item
 $p<q$ and $q\in [3m]\setminus [2m]$.
\end{enumerate}
In both cases, these ESOs only changes the weights of  $(i,j)$ such that $i<j$ and $j\in [3m]\setminus [2m]$. None of
the weights of extra arcs in $E\setminus E_Q$ can be modified by the ESOs stabilizing $QU_n$. So $A$ is a \bform. 

The number of extra arcs of $A$ is $1+2+\cdots+(m-1)=\frac{1}{2}(m-1)m$. By Theorem \ref{thm: Bform number of parameters},
The number of parameters in $A$ is
$$
 \frac{1}{2}(m-1)m+2m-3m+1= \frac{1}{2}(m-1) (m-2).
$$

Given $r\in\{0,1,\ldots, \frac{1}{2}(m-1) (m-2)-1\}$, we can remove $\frac{1}{2}(m-1) (m-2)-r$  extra arcs from the graph $G=G_A$ and keep the remaining graph connected. 
The resulting graph is the graph of an indecomposable \bform\ with $r$ parameters.  So \eqref{Bform parameters} is true for $n\equiv 0\mod 3$. 

\item
When $n=3m-1$, let $G$ be the subgraph of the graph in $n=3m$ case,  obtained by removing the vetex $3m$ and the arc $(m+1,3m)$.  
See the illustrated graph below. 
Similar argument shows that there is a \bform\ $A$ with $G_A=G$, and   \eqref{Bform parameters} is true for $n\equiv 2\mod 3$.
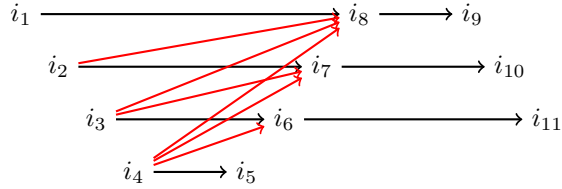
\begin{center}
\begin{tikzpicture}
\node (1) at (0,2.1) {$i_1$};
\node (2) at (0.5,1.4) {$i_2$};
\node (3) at (1,0.7) {$i_3$};
\node (4) at (1.5,0) {$i_4$};
\node (5) at (3,0) {$i_5$};
\node (6) at (3.5,0.7) {$i_6$};
\node (7) at (4,1.4) {$i_7$};
\node (8) at (4.5,2.1) {$i_{8}$};
\node (9) at (6,2.1) {$i_{9}$};
\node (10) at (6.5,1.4) {$i_{10}$};
\node (11) at (7,0.7) {$i_{11}$};
	\foreach \from/\to in {1/8,8/9,2/7,7/10,3/6,6/11,4/5}
            \path[edge] (\from) -- (\to);
	\foreach \from/\to in {2/8,3/8,3/7,4/8,4/7,4/6}
            \path[Redge] (\from) -- (\to);
\end{tikzpicture}
\end{center}
  
\item
When $n=3m-2$, let $G$ be the subgraph of the graph in $n=3m-1$ case,  obtained by removing the vetex $3m-1$ and the arcs $(m+2,3m-1)$ and $(m,m+2)$.  
See the illustrated graph below. 
Similarly, there is a \bform\ $A$ with $G_A=G$, and   \eqref{Bform parameters} is true for $n\equiv 1\mod 3$.
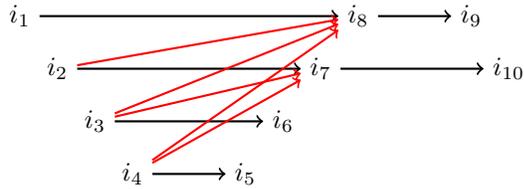
\begin{center}
\begin{tikzpicture}
\node (1) at (0,2.1) {$i_1$};
\node (2) at (0.5,1.4) {$i_2$};
\node (3) at (1,0.7) {$i_3$};
\node (4) at (1.5,0) {$i_4$};
\node (5) at (3,0) {$i_5$};
\node (6) at (3.5,0.7) {$i_6$};
\node (7) at (4,1.4) {$i_7$};
\node (8) at (4.5,2.1) {$i_{8}$};
\node (9) at (6,2.1) {$i_{9}$};
\node (10) at (6.5,1.4) {$i_{10}$};
	\foreach \from/\to in {1/8,8/9,2/7,7/10,3/6,4/5}
            \path[edge] (\from) -- (\to);
	\foreach \from/\to in {2/8,3/8,3/7,4/8,4/7}
            \path[Redge] (\from) -- (\to);
\end{tikzpicture}
\end{center}
\end{enumerate}

The graphs of the above \bform s show that the minimal polynomials of these \bform s are $x^3$. 
\end{proof}

\section{Searches of \Bform s}\label{sect: search Bforms}  


\subsection{Algorithms to search for \Bform s}

We apply the results in Sections 2-4 to get the following efficient  algorithm to obtain the  \bform s under the $B_n$-similarity for a given $n$.

{\bf Algorithm:}
\begin{enumerate}
\item List all   subpermutations $Q$ in $N_n$ (by the set partitions ${\mc P}_Q$ of $[n]$).
\item For each subpermutation $Q$, apply Theorems \ref{thm: Bform arcs} and \ref{thm: long chain deletion}
to filter out a set $S$ of possible extra arcs  of  \Bform s in $QU_n$ and list them in \Border, say, $S:=\{(i_1,j_1)\prec\cdots\prec (i_m,j_m)\}$. 
\item Explore all possible combinations of the above extra arcs that produce \bform s. To do this, let $S_0:=\emptyset$ and start at $p=1$: 
  \begin{enumerate}
      \item 
      Determine whether there exists an $U_n$-similarity operation $g$ composed by ESOs  stabilizing $QU_n$ such that $g$ changes the graph $([n], E_Q\cup S_{p-1}\cup\{(i_p,j_p)\})$ to a graph $([n], E')$ where 
      $$ 
      E'\subseteq E_Q\cup S_{p-1}\cup \{(i,j)\in S: (i,j)\succ (i_p,j_p)\}.$$ 
      If such a $g$ exists, then the arc $(i_p,j_p)$ can be removed by the $g$ operation from the graph of any $A\in QU_n$ whose set of  extra arcs no greater than $(i_p,j_p)$ is $S_{p-1}\cup\{(i_p,j_p)\}$; we let $S_p:=S_{p-1}$.
      Otherwise,  divide the upcoming process into two cases $S_p:=S_{p-1}\cup\{(i_p,j_p)\}$ and $S_p:=S_{p-1}$. 
      \item Increase   $p$   by $1$. If $p\le m$, repeat the preceding process. 
      \item The outcoming sets $S_m$ are the sets of extra arcs of all \bform s in $QU_n$. Apply Theorem \ref{thm: Bform places of parameters} to determine the places of parameters   for each graph type of \Bform s.
  \end{enumerate}
\end{enumerate}

The above algorithm can be restricted to search for only the indecomposable \bform s. Moreover, the algorithm can  be slightly modified to obtain \Bform\ of a given matrix 
$A\in N_n$ after finding $A'\in QU_n$ such that $A\SIM{U_n} A'$ by steps (1) and (2) of the simplification process at the end of Section \ref{sect: preliminary}. 
It is much more efficient than \Balg.

\begin{ex}\label{ex: Bform classification}
Let us search \Bform s in $QU_8$ in which ${\mc P}_Q=123|478|56$. By Theorems \ref{thm: Bform arcs} and \ref{thm: long chain deletion},
the possible extra arcs of a \bform\ in $QU_8$ are listed in \Border\ as follows: \\
\twopage{.5}{.5}
{
$$
(5,7)\prec (2,4)\prec (2,5)\prec (1,4)\prec (1,5).
$$
Let $S$ be the set of these arcs. 
The graph $G_Q$ (in solid arcs) and the  set $S$ (in dashed arcs) are shown on the  right. 
}
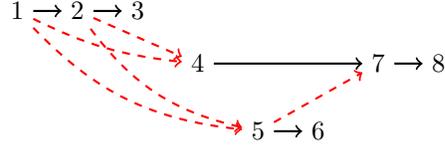
{
\begin{tikzpicture}
\node (1) at (0,1) {$1$};
\node (2) at (0.8,1) {$2$};
\node (3) at (1.6,1) {$3$};
\node (4) at (2.4,0.3) {$4$};
\node (5) at (3.2,-0.6) {$5$};
\node (6) at (4,-0.6) {$6$};
\node (7) at (4.8,0.3) {$7$};
\node (8) at (5.6,0.3) {$8$};
	\foreach \from/\to in {1/2,2/3,4/7,7/8,5/6}
            \draw[edge] (\from) -- (\to);
	\foreach \from/\to in {5/7,2/4}
            \path[Redge, dashed] (\from) -- (\to);
	\foreach \from/\to in {1/4}
           \draw [Redge, dashed] (\from) to [bend right = 10] (\to);
	\foreach \from/\to in {2/5,1/5}
           \draw [Redge, dashed] (\from) to [bend right = 20] (\to);
\end{tikzpicture}
}

Let $S_0:=\emptyset$. 
The arc $(5,7)$ cannot be removed  by a composition of ESOs  stabilizing $QU_8$,
since the only type of ESOs that changes the weight of the arc $(5,7)$ is $\O_{6,7}$ which does not stabilize $QU_8$.
 There are two cases $S_1:=\{(5,7)\}$ and $S_1:=\emptyset$. Consider the case $S_1:=\{(5,7)\}$. Similarly, we can reach one of the outcoming cases $S_2:=\{(5,7),(2,4)\}$ and 
$S_3:=\{(5,7),(2,4),(2,5)\}$. The next   arc in consideration is $(1,4)$, which can be removed as illustrated below (cf. \eqref{to get Bform}): 
$$
\begin{matrix} 
&\O_{12} &\O_{23} &
\\
{\mc P}_Q:57|24|25|14|\cdots 
&-14+13+15 &-13 &=57|24|25|15.
\end{matrix}
$$
Note that  $(1,5)$ created  in the above process satisfies $(1,5)\succ (1,4)$. Now $S_4:= \{(5,7),(2,4),(2,5)\}$.
The next arc $(1,5)$ can be removed as below:
$$
\begin{matrix} 
&\O_{25} &\O_{36} &\O_{47} &\O_{78} &\O_{68} &
\\
{\mc P}_Q:57|24|25|15 
&-15+26+27 &-26 &-27+48 &-48+58 &-58 &=57|24|25.
\end{matrix}
$$
So $S_5:=\{(5,7),(2,4),(2,5)\}$. We get a \bform\ of the type ${\mc P_Q}:57|24|\ul{25}$ in which the underline indicates the place of a parameter. Explicitly,
the $8\times 8$ \bform\ is:
$$
E_{12}+E_{23}+E_{47}+E_{78}+E_{56}+E_{57}+E_{24}+\lambda E_{25}.
$$
\twopage{.38}{.6}
{
The table  on the right lists the forms of \bform s in $QU_8$ for ${\mc P}_Q=123|478|56$. We use ``Y'' (resp. ``N'') to 
mark the presence (resp. absence) of an extra arc, and
``--'' to indicate that the arc can be removed given the combination 
of preceding extra arcs. Totally there are 10 forms of \bform s in the double coset $B_8QB_8$, and 5 of them are indecomposable.
}
{\begin{center}
The \bform s  for ${\mc P}_Q=123|478|56$
\end{center}
\begin{flushright}
\begin{tabular}{|ccccc|l|c|}
\hline 
$57$ &$24$ &$25$ &$14$ &$15$ &type & indecomp.
\\ \hline
Y &Y &Y &-- &-- &$57|24|\ul{25}$ &Yes
\\ \hline
Y&Y &N &--&-- &$57|24$	&Yes 
\\ \hline
Y&N&Y&--&-- &$57|25$	&Yes 
\\ \hline
Y&N&N&--&Y &$57|15$	&Yes 
\\ \hline
Y&N&N&--&N &$57$	&No
\\ \hline
N&Y&Y&--&-- &$24|25$	&Yes
\\ \hline
N&Y&N&--&-- &$24$	&No
\\ \hline
N&N&Y&--&-- &$25$	&No
\\ \hline
N&N&N&Y&-- &$14$	&No
\\ \hline
N&N&N&N&-- &$\emptyset$	&No
\\ \hline
\end{tabular}
\end{flushright}
}
\end{ex}

\subsection{The indecomposable \bform s for $n\le 8$}

In this subsection, we describe the indecomposable \bform s in $N_n$ under the $B_n$-similarity  for $n\le 8$ using their graph types together with underlines indicating nonzero parameters (see Example \ref{ex: Bform classification}). The classifications for $n\le 6$ have been done by Kobal \cite{Kobal2005} and Chen et al \cite{Chen2016} (see Theorem \ref{thm: Bform n le 6}).  We apply MAPLE programs to filter out possible extra arcs using the algorithm in the preceding subsection and obtain all classifications for $n\le 8$. 

\begin{thm}[Kobal \cite{Kobal2005}, Chen et al \cite{Chen2016}]\label{thm: Bform n le 6} 
The indecomposable \bform s in $N_n$ under the $B_n$-similarity for $n\le 6$ are listed by their  graph types together with underlines indicating nonzero parameters
as follows (29 forms, separated by commas):
\begin{multicols}{3}
\noindent
$1: \ \emptyset,$ \\
$12: \ \emptyset,$ \\
$123: \ \emptyset,$ \\
$ 1 2 3 4: \ \emptyset,$ \\
$ 1 2|3 4: \ 13,$ \\
$ 1 2 3 4 5: \ \emptyset,$ \\
$ 1 2 3|4 5: \ 24,$ \\
$ 1 2 5|3 4: \ 13,$ \\
$ 1 4 5| 2 3: \ 24,$ \\
$ 1 2|3 4 5: \ 13,$ \\
$ 1 2 3 4 5 6: \ \emptyset,$ \\
$ 1 2 3 4|5 6: \ 35,$ \\
$ 1 2 3 6|4 5: \ 24,$ \\
$ 1 2 5 6| 3 4: \ 35 | \ul{13}, 35,  13,$ \\
$ 1 4 5 6| 2 3 : \ 24,$ \\
$ 1 3 4| 2 5 6: \ 35,$ \\
$ 1 4 5| 2 3 6: \ 24,$ \\
$ 1 5 6| 2 3 4: \ 35,$ \\
$1 2 3|4 5 6: \ 24, \ 14,$ \\
$1 2 4| 3 5 6: \ 13,$ \\
$1 2 5| 3 4 6: \ 13,$ \\
$1 2 6| 3 4 5: \ 13,$ \\
$1 2 |3 4 5 6: \ 13,$ \\
$1 2| 3 4| 5 6 : \ 35 | 13,$ \\
$1 2| 3 6| 4 5 : \ 13 | 14,$ \\
$ 1 4| 2 3| 5 6 : \ 25 | 15.$
\end{multicols}
\end{thm}

\begin{rem}
For $n=6$, Theorem \ref{thm: Bform n le 6} lists 19 forms instead of 18 forms shown in \cite[Theorem 2.2]{Chen2016}, since we use nonzero parameters in our classifications. 
\end{rem}

The results for $n=7$ are as follows, including the $8$ forms with a parameter  discovered in \cite[Theorem 2.3]{Chen2016}.

\begin{thm}\label{thm: Bform n=7} 
The indecomposable \bform s in $N_n$ under the $B_n$-similarity for $n=7$ are of the graph types (85 forms in 58 subpermutations, separated by commas): 
\begin{multicols}{3}
\noindent
$1 2 3 4 5 6 7:$   $\emptyset ,$ \\
$1 4 5 6 7| 2 3 :$   $ 24 ,$ \\
$1 2 5 6 7| 3 4:$   $ 35  |  \underline{13},$  $35,$  $13 ,$ \\
$1 2 3 6 7|4 5:$   $ 46  |  \underline{24},$  $46,$  $ 24 ,$ \\
$1 2 3 4 7|5 6:$   $ 35  ,$ \\
$1 2 3 4 5|6 7:$   $ 46  ,$ \\
$1 5 6 7| 2 3 4  :$   $ 35,$  $ 25 ,$ \\
$1 4 6 7| 2 3 5 :$   $ 24 ,$ \\
$1 4 5 7| 2 3 6 :$   $ 24 ,$ \\
$1 4 5 6| 2 3 7 :$   $ 24 ,$ \\
$1 3 6 7|2 4 5 :$   $ 46 ,$ \\
$1 3 4 7|2 5 6 :$   $ 35 ,$ \\
$1 3 4 5|2 6 7 :$   $ 46 ,$ \\
$1 2 6 7|3 4 5 :$   $ 46  |  \underline{13},$  $    46,$  $13 ,$ \\
$1 2 5 7| 3 4 6:$   $ 13 ,$ \\
$1 2 5 6| 3 4 7:$   $ 35  |  \underline{13},$  $    35,$  $ 13 ,$ \\
$1 2 4 7| 3 5 6:$   $ 13 ,$ \\
$1 2 4 5| 3 6 7:$   $ 46 ,$ \\
$1 2 3 7| 4 5 6:$   $ 24,$  $    14 ,$ \\
$1 2 3 6|4 5 7:$   $ 24 ,$ \\
$1 2 3 5|4 6 7:$   $ 24 ,$ \\
$1 2 3 4|5 6 7:$   $ 35,$  $    25 ,$ \\
$1 2 3|4 5 6 7:$   $ 24,$  $    14  ,$ \\
$1 2 4| 3 5 6 7:$   $ 13 ,$ \\
$1 2 5| 3 4 6 7:$   $ 13 ,$ \\
$1 2 6| 3 4 5 7 :$   $ 13 ,$ \\
$1 2 7| 3 4 5 6 :$   $ 13 ,$ \\
$1 3 4| 2 5 6 7 :$   $ 35 ,$ \\
$1 4 5| 2 3 6 7 :$   $ 46  |  \underline{24},$  $    46,$  $    24 ,$ \\
$1 5 6| 2 3 4 7 :$   $ 35 ,$ \\
$1 6 7| 2 3 4 5 :$   $ 46 ,$ \\
$1 2 3|4 5 | 6 7:$   $ 46  |  24,$  $    46  |  14 ,$ \\
$1 2 3|4 7 | 5 6:$   $ 24 | 25 ,$ \\
$1 2 4| 3 5| 6 7:$   $ 36  |  13 ,$ \\
$1 2 5| 3 4| 6 7:$   $ 36  |  26  |  \underline{13},$    $ 36  |  26,$   $ 36  |  13,$   $ 26  |  13 ,$ \\
$1 2 6 | 3 4 | 5 7 :$   $ 35  |  13 ,$ \\
$1 2 7| 3 4| 5 6 :$   $ 35  |  13 ,$ \\
$1 2 7| 3 6| 4 5 :$   $ 13  |  14 ,$ \\
$1 3 4| 2 5| 6 7 :$   $ 36  |  26 ,$ \\
$1 4 5| 2 3| 6 7 :$   $ 46  |  24,$  $    24  |  16 ,$ \\
$1 4 6| 2 3| 5 7 :$   $ 24  |  15 ,$ \\
$1 4 7| 2 3| 5 6 :$   $ 24  |  25  |  \underline{15},$  $    24  |  25,$  $    24  |  15,$  $    25  |  15 ,$ \\
$1 5 6| 2 3| 4 7 :$   $ 24  |  25 ,$ \\
$1 6 7| 2 3| 4 5 :$   $ 46  |  24 ,$ \\  
$1 6 7| 2 5| 3 4 :$   $ 36  |  26 ,$ \\
$1 2 |3 4 5 6 7:$   $ 13  ,$ \\
$1 5 | 2 3 4|6 7 :$   $ 36  |  16 ,$ \\
$1 4 | 2 3 7|5 6 :$   $ 25  |  15 ,$ \\
$1 3 | 2 4 7|5 6 :$   $ 25  |  15 ,$ \\
$1 3 | 2 6 7|4 5 :$   $ 46  |  14 ,$ \\
$1 2 | 3 4 5|6 7 :$   $ 46  |  13 ,$ \\
$1 2 | 3 4 7|5 6 :$   $ 35  |  13,$  $    35  |  15 ,$ \\
$1 2 | 3 6 7|4 5 :$   $ 46  |  13  | \underline{14},$  $  46  |  13,$  $  46  |  14,$  $  13  |  14 ,$ \\
$1 2 | 3 7|4 5 6 :$   $ 13  |  14 ,$ \\
$1 2 | 3 6|4 5 7 :$   $ 13  |  14 ,$ \\
$1 2 | 3 5|4 6 7 :$   $ 13  |  14 ,$ \\
$1 2 | 3 4|5 6 7 :$   $ 35  |  13,$  $    13  |  15 ,$ \\
$1 4 | 2 3|5 6 7 :$   $ 25  |  15.$
\end{multicols}
\end{thm}

\begin{thm}\label{thm: Bform n=8} 
For $n=8$, there are 481 forms (in 245 subpermutations) of  indecomposable \bform s in $N_n$ under the $B_n$-similarity. The graph types of these forms will be listed  in the Appendix section. 
\end{thm}

The number of subpermutations in $N_n$ equals the number of partitions of $[n]$, which  is called a Bell  or exponential number. 
The first few Bell numbers starting at $n=1$ are: 
$$1, 2, 5, 15, 52, 203, 877, 4140, 21147, \ldots$$ 
Many properties of the Bell numbers have been studied (cf. \href{http://oeis.org}{http://oeis.org}). 
Both Theorem \ref{thm: combine Bforms} and the Bell numbers imply that
the numbers of indecomposable \bform s in $N_n$  grow in a rate greater than any exponential function of $n$.

\subsection{Create new \bform s}

Theorem \ref{thm: combine Bforms} can be used to obtain new \bform s. Let $A_1\in Q_1U_p$ and $A_2\in Q_2U_q$ be \bform s.
Theorem \ref{thm: combine Bforms} claims that if $\mtx{Q_1 &Q_{12}\\0 &Q_2}\in N_{p+q}$ is a subpermutation, then
$\mtx{A_1 &Q_{12}\\0 &A_2}$ is a \bform\ in $N_{p+q}$. 
Here we rephrase Theorem \ref{thm: combine Bforms} in the language of graphs.
Let $G_{A_2}+p$ denote the graph with the vertex set $\{1+p,\ldots,q+p\}$ and the edge set $\{(i+p,j+p): (i,j)\in E_{A_2}\}$.
If we add some arcs from chain heads of $G_{A_1}$ to chain tails of $G_{A_2}+p$  such that each involving vertex is on at most one such arc, then the resulting graph represents a \bform\ in $N_{p+q}$. 
 The following are two examples. 

\begin{ex}\label{ex: combine bforms}
Let $A_1=A_2$  be \Bform s of the graph type $12|34:13$. So ${\mc P}_{Q_1}={\mc P}_{Q_2}=12|34$. 
By Theorem \ref{thm: combine Bforms}, we can obtain \bform s by adding to the graph $G_{A_1}\cup (G_{A_2}+4)$ 
some arcs from the chain heads $2$ and  $4$ of $G_{A_1}$ to
the chain tails $5$ and $7$ of $G_{A_2}+4$ such that each vertex is on at most one arc. The illustrated graph is as below, in which dashed arcs are possible arcs added to the graph $G_{A_1}\cup (G_{A_2}+4)$:
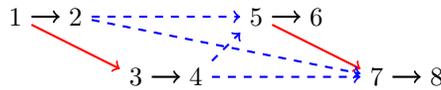
\begin{center}
\begin{tikzpicture}
\node (1) at (0,0.8) {$1$};  
\node (2) at (0.8,0.8) {$2$};
\node (3) at (1.6,0) {$3$};
\node (4) at (2.4,0) {$4$}; 
\node (5) at (3.2,0.8) {$5$};
\node (6) at (4.0,0.8) {$6$};
\node (7) at (4.8,0) {$7$};
\node (8) at (5.6,0) {$8$};
	\foreach \from/\to in {1/2,5/6,3/4,7/8}
            \path[edge] (\from) -- (\to);
	\foreach \from/\to in {1/3,5/7}
            \path[Redge] (\from) -- (\to);
	\foreach \from/\to in {2/5, 2/7, 4/5, 4/7}
            \path[Bedge, dashed] (\from) -- (\to);
\end{tikzpicture}
\end{center}
There are 6 indecomposable \bform s obtained from this way:
$1256|34|78:57|13$, $1278|34|56:57|13$, $12|3456|78:57|13$, $12|3478|56:57|13$, $1256|3478:57|\ul{13}$, $1278|3456:57|\ul{13}$. 
All of them can be found in Theorem \ref{thm: Bform n=8}.  Note that a parameter is added in each of the last two cases for the general forms. 
The direct sum $A_1\oplus A_2$ is the only non-indecomposable \bform\ given by Theorem \ref{thm: combine Bforms} here.
\end{ex}

\begin{ex}
Let $A_1\in Q_1U_6$ and  $A_2\in Q_2U_3$  be \Bform s of the graph types  $1 2| 3 6| 4 5 :  13 | 14$ and $13|2:\emptyset$, respectively.
The \bform s $\mtx{A_1 &Q_{12}\\ &A_2}$ constructed in Theorem \ref{thm: combine Bforms} are constructed by adding the following possible dashed arcs to
the graph $G_{A_1}\cup (G_{A_2}+6)$ such that each vertex is on at most one such arc:
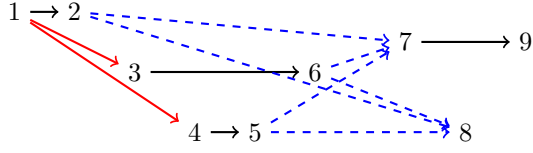
\begin{center}
\begin{tikzpicture}
\node (1) at (0,1.6) {$1$};  
\node (2) at (0.8,1.6) {$2$};
\node (3) at (1.6,0.8) {$3$};
\node (4) at (2.4,0) {$4$}; 
\node (5) at (3.2,0) {$5$};
\node (6) at (4.0,0.8) {$6$};
\node (7) at (5.2,1.2) {$7$};
\node (8) at (6,0) {$8$};
\node (9) at (6.8,1.2) {$9$};
	\foreach \from/\to in {1/2,3/6,4/5,7/9} 
            \path[edge] (\from) -- (\to);
	\foreach \from/\to in {1/3,1/4} 
            \path[Redge] (\from) -- (\to);
	\foreach \from/\to in {2/7, 2/8, 6/7, 6/8, 5/7,5/8} 
            \path[Bedge, dashed] (\from) -- (\to);
\end{tikzpicture}
\end{center}
There are 6 indecomposable \bform s constructed in this way: $1279|368|45:13|14$, $1279|36|458:13|14$, 
$128|3679|45:13|14$, $128|36|4579:13|14$, $12|3679|458:13|14$, $12|368|4579:13|14$. 

Similarly, the \bform s $\mtx{A_2 &Q_{12}'\\ &A_1}$ constructed in Theorem \ref{thm: combine Bforms} are constructed by adding the following possible dashed arcs to
the graph $G_{A_2}\cup (G_{A_1}+3)$ such that each vertex is on at most one such arc:
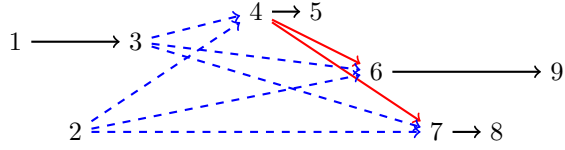
\begin{center}
\begin{tikzpicture}
\node (1) at (-3.2,1.2) {$1$};
\node (2) at (-2.4,0) {$2$};
\node (3) at (-1.6,1.2) {$3$};
\node (4) at (0,1.6) {$4$};  
\node (5) at (0.8,1.6) {$5$};
\node (6) at (1.6,0.8) {$6$};
\node (7) at (2.4,0) {$7$}; 
\node (8) at (3.2,0) {$8$};
\node (9) at (4.0,0.8) {$9$};
	\foreach \from/\to in {1/3,4/5,6/9,7/8} 
            \path[edge] (\from) -- (\to);
	\foreach \from/\to in {4/6,4/7} 
            \path[Redge] (\from) -- (\to);
	\foreach \from/\to in {2/4,2/6,2/7,3/4,3/6,3/7} 
            \path[Bedge, dashed] (\from) -- (\to);
\end{tikzpicture}
\end{center}
There are also 6 indecomposable \bform s constructed in this way:
$1345|269|78:46|47$, $1345|278|69:46|47$, 
$1369|245|78:46|47$, $1369|278|45:46|47$, 
$1378|245|69:46|47$, $1378|269|45:46|47$.
\end{ex}

\subsection{The $B_n$-similarity of upper triangular matrices} 

Sylveseter's theorem (cf. \cite[Theorem 2.4.4.1]{Horn2013}) says that if 
$M\in M_p$ and $N\in M_q$ have no eigenvalue in common, then
 the equation $MX-XN=R$ has a unique solution $X\in M_{p,q}$ for each $R\in M_{p,q}$, that is,
\beq
\mtx{M &R\\0 &N}=\mtx{I_p &X\\0 &I_q}^{-1}\mtx{M &0\\0 &N}\mtx{I_p &X\\0 &I_q}\SIM{U_{p+q}}\mtx{M &0\\0 &N}.
\eeq
Similarly, every $n\times n$ upper triangular matrix $A=[a_{ij}]$ is $B_n$-similar to the matrix $C=[c_{ij}]$ such that
$c_{ij}=a_{ij}$ if $a_{ii}=a_{jj}$, and $c_{ij}=0$ otherwise;  $C$ is permutation similar to a direct sum of  matrices of the form $\lambda I_{k}+C'$ for $k\in [n]$ and
$C'\in N_{k}$. 
See \cite{Roitman1978} or   \cite[Section 1]{Thijsse1997} for more details.
The $B_n$-similarity problem of  upper triangular $A$ is transformed to the  upper triangular similarity problems of nilpotent upper triangular  matrices. 

In \cite[Theorem 1.5]{Thijsse1997}, Thijsse showed that if an $n\times n$ upper triangular matrix $A$ satisfies one of the following two conditions:
\begin{enumerate}
\item $A$ is nonderogatory;
\item $\dim \ker(A-\lambda I)^2=\dim \ker(A-\lambda I)^3$ for each $\lambda\in\C$.
\end{enumerate}
Then   $A$ is $B_n$-similar to a matrix which is permutation similar to a direct sum of Jordan blocks. 
We provide a new proof here. 
By the argument in the preceding paragraph, it suffices to consider
the case $A\in N_n$: 
\begin{itemize}
\item Condition (1) means that $A\in JU_n$ where $J$ is the nilpotent Jordan block of size $n$. The only \bform \   in $JU_n$ is $J$.
\item
Condition (2) means that each Jordan block of $A$ has size no more than 2. So $A\SIM{B_n} A'\in QU_n$ in which $Q\in N_n$ is a subpermutation, $A^2=(A')^2=0$, 
and each extra arc $(i,j)\in E_{A'}\setminus E_{Q}$ of $A'$ has $j$ a chain head of a chain of $G_Q$. So $(i,j)$ is not in \Bform \ of $A$. Therefore, \Bform\ of $A$ is exactly $Q$, which is permutation similar to a direct sum of nilpotent Jordan blocks of sizes one or two.
\end{itemize}

Another observation is \cite[2.5.P49]{Horn2013} which states that: an  $n\times n$ upper triangular matrix $A$ is similar to a diagonal matrix if and only if it is $B_n$-similar to a diagonal matrix.  
In particular,  if $A$ has distinct diagonal entries, then $A$ is $B_n$-similar to its diagonal. 
However,  the result is not quite useful for  nilpotent upper triangular matrices since the only diagonalizable nilpotent upper triangular matrix is the zero matrix.

\section*{Appendix}

As a complement to Theorem \ref{thm: Bform n=8}, the list of indecomposable \bform s in $N_n$ under the $B_n$-simlarity for $n=8$  is as follows (481 forms in 245 subpermutations):
\begin{multicols}{3}
\noindent
$1 2 3 4 5 6 7 8:$ $\emptyset, $ \newline
             $1 2 | 3 4 5 6 7 8:$ $1 3,$ \newline
             $1 4 5 6 7 8 | 2 3:$ $2 4,$ \newline
         $1 2 5 6 7 8 | 3 4:$ $35|\ul{13}, 35, 13,$ \newline
     $1 2 3 6 7 8 | 4 5:$ $46|\ul{24}|\ul{14},$ $46|\ul{24},$ $46,$ $24,$ \newline
 $1 2 3 4 7 8|5 6:$ $57 | \ul{35}, 57, 35,$ \newline
             $1 2 3 4 5 8 | 6 7:$ $4 6,$ \newline
             $1 2 3 4 5 6 | 7 8:$ $5 7,$ \newline
         $1 2 3 | 4 5 6 7 8:$ $2 4, 1 4,$ \newline
             $1 2 4 | 3 5 6 7 8:1 3,$ \newline
             $1 2 5 | 3 4 6 7 8:$ $1 3,$ \newline
             $1 2 6 | 3 4 5 7 8:$ $1 3,$ \newline
             $1 2 7 | 3 4 5 6 8:$ $1 3,$ \newline
             $1 2 8 | 3 4 5 6 7:$ $1 3,$ \newline
             $1 3 4 | 2 5 6 7 8:$ $3 5,$ \newline
         $1 4 5 | 2 3 6 7 8:$ $46|\ul{24},$ $46,$ $24,$ \newline
     $1 5 6 | 2 3 4 7 8:$ $57|\ul{35}, 57, 35,$ \newline
             $1 6 7 | 2 3 4 5 8:$ $4 6,$ \newline
             $1 7 8 | 2 3 4 5 6:$ $5 7,$ \newline
         $1 5 6 7 8| 2 3 4:$ $3 5, 25,$ \newline
             $1 4 6 7 8 | 2 3 5 :$ $ 2 4,$ \newline
             $1 4 5 7 8 | 2 3 6 :$ $2 4,$ \newline
             $1 4 5 6 8 | 2 3 7:$ $2 4,$ \newline
             $1 4 5 6 7 | 2 3 8:$ $2 4,$ \newline
             $1 3 6 7 8 | 2 4 5:$ $4 6,$ \newline
         $1 3 4 7 8 | 2 5 6:$ $57|\ul{35},$ $57,$ $35,$ \newline
             $ 1 3 4 5 8 | 2 6 7 :$ $4 6,$ \newline
             $ 1 3 4 5 6 | 2 7 8 :$ $5 7,$ \newline
     $1 2 6 7 8 | 3 4 5:$ $46|\ul{13},$ $46,$ $36|\ul{13},$ $36,$ $13,$ \newline
         $1 2 5 7 8 | 3 4 6:$ $35|\ul{13}, 35, 13,$ \newline
         $1 2 5 6 8 | 3 4 7:$ $35|\ul{13}, 35, 13,$ \newline
         $1 2 5 6 7 | 3 4 8:$ $35|\ul{13}, 35, 13,$ \newline
         $1 2 4 7 8 | 3 5 6:$ $57|\ul{13}, 57, 13,$ \newline
             $1 2 4 5 8 | 3 6 7:$ $4 6,$ \newline
             $1 2 4 5 6 | 3 7 8:$ $5 7,$ \newline
     $1 2 3 7 8 | 4 5 6:$ $57|\ul{24},$ $57|\ul{14},$ $57,$ $24,$ $14,$ \newline
     $1 2 3 6 8 | 4 5 7:$ $2 4,$ \newline
     $1 2 3 6 7 | 4 5 8:$ $46|\ul{24}, 46, 24,$ \newline
             $1 2 3 5 8 | 4 6 7:$ $2 4,$ \newline
             $1 2 3 5 6 | 4 7 8:$ $5 7,$ \newline
         $1 2 3 4 8 | 5 6 7:$ $3 5, 2 5,$ \newline
             $1 2 3 4 7 | 5 6 8:$ $3 5,$ \newline
             $1 2 3 4 6 | 5 7 8:$ $3 5,$ \newline
         $1 2 3 4 5 | 6 7 8:$ $4 6, 3 6,$ \newline
     $1 2 3 4 | 5 6 7 8:$ $35, 25, 15,$ \newline
         $1 2 3 5 | 4 6 7 8:$ $24, 14,$ \newline
         $1 2 3 6 | 4 5 7 8:$ $24, 14,$ \newline
         $1 2 3 7 | 4 5 6 8:$ $24, 14,$ \newline
         $1 2 3 8 | 4 5 6 7:$ $24, 14,$ \newline
         $1 2 4 5 | 3 6 7 8:$ $4 6, 13,$ \newline
             $1 2 4 6 | 3 5 7 8:$ $1 3,$ \newline
             $1 2 4 7 | 3 5 6 8:$ $1 3,$ \newline
             $1 2 4 8 | 3 5 6 7:$ $1 3,$ \newline
     $1 2 5 6 | 3 4 7 8:$ $57|\ul{35}|\ul{13},$ $57|\ul{35},$ $57|\ul{13},$ $57,$ $35|\ul{13},$  $35,$ $13,$ \newline
             $1 2 5 7 | 3 4 6 8:$ $1 3,$ \newline
             $1 2 5 8 | 3 4 6 7:$ $1 3,$ \newline
         $1 2 6 7 | 3 4 5 8:$ $4 6|\ul{13}, 46, 13,$ \newline
             $1 2 6 8 | 3 4 5 7:$ $1 3,$ \newline
         $1 2 7 8 | 3 4 5 6:$ $57|\ul{13}, 57, 13,$ \newline
         $1 3 4 5 | 2 6 7 8:$ $4 6,  3 6,$ \newline
             $1 3 4 6 | 2 5 7 8:$ $3 5,$ \newline
             $1 3 4 7 | 2 5 6 8:$ $3 5,$ \newline
             $1 3 4 8 | 2 5 6 7:$ $3 5,$ \newline
             $1 3 5 6 | 2 4 7 8:$ $5 7,$ \newline
             $1 3 6 7 | 2 4 5 8:$ $4 6,$ \newline
             $1 3 7 8 | 2 4 5 6:$ $5 7,$ \newline
         $1 4 5 6 | 2 3 7 8:$ $57|\ul{24}, 5 7, 24,$ \newline
             $1 4 5 7 | 2 3 6 8:$ $2 4,$ \newline
         $1 4 5 8 | 2 3 6 7:$ $4 6|\ul{24}, 46, 24,$ \newline
             $1 4 6 7 | 2 3 5 8:$ $2 4,$ \newline
             $1 4 7 8 | 2 3 5 6:$ $5 7,$ \newline
         $1 5 6 7 | 2 3 4 8:$ $3 5, 25,$ \newline
             $1 5 6 8 | 2 3 4 7:$ $3 5,$ \newline
             $1 5 7 8 | 2 3 4 6:$ $3 5,$ \newline
         $1 6 7 8 | 2 3 4 5:$ $4 6, 36,$ \newline
    $1 2 | 3 4 | 5 6 7 8:$ $35|13, 13|15,$ \newline
        $1 2 | 3 5 | 4 6 7 8:$ $1 3 | 1 4,$ \newline
        $1 2 | 3 6 | 4 5 7 8:$ $1 3 | 1 4,$ \newline
        $1 2 | 3 7 | 4 5 6 8:$ $1 3 | 1 4,$ \newline
        $1 2 | 3 8 | 4 5 6 7:$ $1 3 | 1 4,$ \newline
    $1 2  | 3 6 7 8 | 4 5:$ $46|13|\ul{14},$ $46|13,$ $46|14,$ $13|14,$ \newline
$1 2 | 3 4 7 8 | 5 6:$ $57|\ul{35}|13,$ $57|\ul{35}|15,$ $57|13,$ $57|15,$ $35|13,$ $35|15,$ \newline
    $1 2 | 3 4 5 8 | 6 7:$ $46|13, 46|16,$ \newline
    $1 2 | 3 4 5 6 | 7 8:$ $5 7|13,$ \newline
$1 2 | 3 4 5 | 6 7 8:$ $46|13, 36|13, 36|16,$ \newline
    $1 2 | 3 4 6 | 5 7 8:$ $35|13, 35|15,$ \newline
    $1 2 | 3 4 7 | 5 6 8:$ $35|13, 35|15,$ \newline
    $1 2 | 3 4 8 | 5 6 7:$ $35|13, 35|15, 13|15,$ \newline
    $1 2 | 3 5 6 | 4 7 8:$ $57|13, 57|14,$ \newline
        $1 2 | 3 5 8 | 4 6 7:$ $1 3 | 1 4,$ \newline
    $1 2 | 3 6 7 | 4 5 8:$ $46|13|\ul{14},$ $46|13,$ $46|14,$ $13|14,$ \newline
        $1 2 | 3 6 8 | 4 5 7:$ $1 3 | 1 4,$ \newline
    $1 2 | 3 7 8 | 4 5 6:$ $57|13|\ul{14},$ $57|13,$ $57|14,$ $13|14,$ \newline
        $1 3 | 2 6 7 8 | 4 5:$ $4 6| 14, $ \newline
    $1 3 | 2 4 7 8 | 5 6:$ $57|\ul{25}|15,$ $57|15,$ $25|15,$ \newline
        $1 3 | 2 4 5 8 | 6 7:$ $4 6|1 6,$ \newline
    $1 3 | 2 4 5 | 6 7 8:$ $2 6|16,$ \newline
        $1 3 | 2 4 6 | 5 7 8:$ $2 5 | 1 5,$ \newline
        $1 3 | 2 4 7 | 5 6 8:$ $2 5 | 1 5,$ \newline
        $1 3 | 2 4 8 | 5 6 7:$ $2 5 | 1 5,$ \newline
        $1 3 | 2 5 6 | 4 7 8:$ $5 7 | 1 4,$ \newline
        $1 3 | 2 6 7 | 4 5 8:$ $4 6 | 1 4,$ \newline
        $1 3 | 2 7 8 | 4 5 6:$ $5 7 | 1 4,$ \newline
        $1 4 | 2 3 | 5 6 7 8:$ $2 5 | 1 5,$ \newline
    $1 4 | 2 3 7 8 | 5 6:$ $57|\ul{25}|15,$ $57|15,$ $25|15,$ \newline
        $1 4 | 2 3 5 8 | 6 7:$ $3 6 | 1 6,$ \newline
    $1 4 | 2 3 5 | 6 7 8:$ $2 6| 1 6,$ \newline
        $1 4 | 2 3 6 | 5 7 8:$ $2 5 | 1 5,$ \newline
        $1 4 | 2 3 7 | 5 6 8:$ $2 5 | 1 5,$ \newline
        $1 4 | 2 3 8 | 5 6 7:$ $2 5 | 1 5,$ \newline
        $1 5 | 2 3 4 8 | 6 7:$ $3 6 | 1 6,$ \newline
    $1 5 | 2 3 4 | 6 7 8:$ $3 6|16, 26|16,$ \newline
        $1 6 | 2 3 4 5 | 7 8:$ $4 7 | 17,$ \newline
    $1 6 7 8 | 2 3 | 4 5:$ $46|2 4, 2 6|24,$ \newline
        $1 5 7 8 | 2 3 | 4 6:$ $2 4 | 2 5,$ \newline
        $1 5 6 8 | 2 3 | 4 7:$ $2 4 | 2 5,$ \newline
        $1 5 6 7 | 2 3 | 4 8:$ $2 4 | 2 5,$ \newline
$1 4 7 8 | 2 3 | 5 6:$ $57|24|\ul{25}|\ul{15},$ $57|24|\ul{25},$ $57|24|\ul{15},$ $57|24,$ $57|25|\ul{15},$ $57|25,$ $24|25|\ul{15},$ $24|25,$ $24|15,$ $25|15,$ \newline
    $1 4 6 8 | 2 3 | 5 7:$ $24|15,$ \newline
    $1 4 6 7 | 2 3 | 5 8:$ $2 4| 1 5,$ \newline
$1 4 5 8 | 2 3 | 6 7:$ $46|24|\ul{16},$ $46|24,$ $46|26,$ $24|16,$ \newline
    $1 4 5 7 | 2 3 | 6 8:$ $2 4 | 1 6,$ \newline
$1 4 5 6 | 2 3 | 7 8:$ $57|24, 24|17,$ \newline
$1 4 5 | 2 3 | 6 7 8:$ $4 6|24,$ $24|26|\ul{16},$ $24|26,$ $24|16,$ $26|16,$ \newline
    $1 4 6 | 2 3 | 5 7 8:$ $24|25|\ul{15},$ $24|25,$ $24|15,$ $25|15,$ \newline
    $1 4 7 | 2 3 | 5 6 8:$ $24|25|\ul{15},$ $24|25,$ $24|15,$ $25|15,$ \newline
    $1 4 8 | 2 3 | 5 6 7:$ $24|25|\ul{15},$ $24|25,$ $24|15,$ $25|15,$ \newline
    $1 5 6 | 2 3 | 4 7 8:$ $57|24|\ul{25},$ $57|24,$ $57|25,$ $24|25,$ \newline
    $1 6 7 | 2 3 | 4 5 8:$ $46|24, 46|26,$ \newline
    $1 7 8 | 2 3 | 4 5 6:$ $57|24,$ \newline
    $1 3 7 8 | 2 4 | 5 6:$ $57|25,$ \newline
        $1 3 5 8 | 2 4 | 6 7:$ $36|26,$ \newline
        $1 5 6 | 2 4 | 3 7 8:$ $57|25,$ \newline
        $1 6 7 | 2 4 | 3 5 8:$ $36|26,$ \newline
        $1 6 7 8 | 2 5 | 3 4:$ $3 6 | 2 6,$ \newline
        $1 3 4 8 | 2 5 | 6 7:$ $3 6 | 2 6,$ \newline
    $1 3 4 | 2 5 | 6 7 8:$ $36|26,$ \newline
        $1 6 7 | 2 5 | 3 4 8:$ $36|26,$ \newline
        $1 3 4 5 | 2 6 | 7 8:$ $4 7 | 2 7,$ \newline
        $1 7 8 | 2 6 | 3 4 5:$ $4 7 | 2 7,$ \newline
           $1 2 7 8 | 3 4 | 5 6:$ $57|35|\ul{13},$ $57|35,$ $57|13|\ul{15},$ $57|13,$ $35|13,$ $37|15,$ \newline
$1 2 6 8 | 3 4 | 5 7:$ $3 5 | 1 3,$ \newline
$1 2 6 7 | 3 4 | 5 8:$ $35|36|\ul{13},$ $35|36,$ $35|13,$  \newline
           $1 2 5 8 | 3 4 | 6 7:$ $35|36|\ul{26}|\ul{13},$ $35|36|\ul{26},$ $35|36|\ul{13},$ $35|36,$ $35|26|\ul{13},$ $35|26,$ $36|26|\ul{13},$ $36|26,$ $36|13,$ $26|13,$ \newline
           $1 2 5 7 | 3 4 | 6 8:$ $35|26|\ul{13},$ $35|26,$ $36|13,$  \newline
       $1 2 5 6 | 3 4 | 7 8:$ $57|35|\ul{13},$ $57|35,$ $57|13,$ $35|27|\ul{13},$ $35|27,$ $37|13,$ \newline
$1 2 5 | 3 4 | 6 7 8:$ $36|26|\ul{13},$ $36|26,$ $36|13,$ $26|13,$ $13|16,$  \newline
    $1 2 6 | 3 4 | 5 7 8:$ $35|13,$ $13|15,$ \newline
    $1 2 7 | 3 4 | 5 6 8:$ $35|13,$ $13|15,$ \newline
    $1 2 8 | 3 4 | 5 6 7:$ $35|13,$ $13|15,$ \newline
    $1 5 6 | 2 7 8 | 3 4:$ $5 7 | 3 5,$ \newline
        $1 5 8 | 2 6 7 | 3 4:$ $3 5 | 3 6,$ \newline
    $1 6 7 | 2 5 8 | 3 4:$ $35|36|\ul{26},$ $35|36,$ $35|26,$ $36|26,$  \newline
    $1 6 8 | 2 5 7 | 3 4:$ $3 5 | 2 6,$ \newline
$1 7 8 | 2 5 6 | 3 4:$ $57|35,$ $35|27,$ \newline
$1 2 7 8 | 3 5 | 4 6:$ $47|13|\ul{14},$ $47|13,$ $37|14,$  \newline
    $1 2 4 8 | 3 5 | 6 7:$ $36|13,$ \newline
$1 2 4 | 3 5 | 6 7 8:$ $36|13,$ \newline
        $1 2 6 | 3 5 | 4 7 8:$ $1 3 | 1 4,$ \newline
        $3 5 | 1 2 7 | 4 6 8:$ $1 3 | 1 4,$ \newline
        $3 5 | 1 2 8 | 4 6 7:$ $1 3 | 1 4,$ \newline
$1 2 7 8 | 3 6 | 4 5:$ $47|37|\ul{13}|\ul{14},$ $47|37|\ul{13},$ $47|37|\ul{14},$ $47|37,$ $47|13|\ul{14},$ $47|13,$ $37|\ul{13}|14,$ $37|14,$ $13|14,$ \newline
    $1 2 4 5 | 3 6 | 7 8:$ $4 7 | 3 7,$ \newline
        $1 2 7 | 3 6 | 4 5 8:$ $1 3 | 1 4,$ \newline
        $1 2 8 | 3 6 | 4 5 7:$ $1 3 | 1 4,$ \newline
        $1 4 5 | 2 7 8 | 3 6:$ $4 7 | 3 7,$ \newline
        $1 7 8 | 2 4 5 | 3 6:$ $4 7 | 3 7,$ \newline
        $1 2 8 | 3 7 | 4 5 6:$ $1 3 | 1 4,$ \newline
$1 2 3 8 | 4 5 | 6 7:$ $46|24,$ $46|14,$ \newline
    $ 1 2 3 7 | 4 5 | 6 8:$ $4 6 | 2 4,$ \newline
$1 2 3 6 | 4 5 | 7 8:$ $47|37|\ul{24},$ $47|37,$ $47|24,$ $37|24,$ \newline
           $1 2 3 | 4 5 | 6 7 8:$ $46|24,$ $46|14,$ $24|26,$ \newline
    $1 2 6 | 3 7 8 | 4 5:$ $47|13|\ul{14},$ $47|13,$ $47|14,$ $13|14,$ \newline
    $1 2 7 | 3 6 8 | 4 5:$ $46|13|\ul{14},$ $46|13,$ $46|14,$ $13|14,$ \newline
    $1 2 8 | 3 6 7 | 4 5:$ $46|13|\ul{14},$ $46|13,$ $46|14,$ $13|14,$ \newline
    $1 3 6 | 2 7 8 | 4 5:$ $47|37|\ul{14},$ $47|37,$ $47|14,$ $37|14,$ \newline
        $1 3 7 | 2 6 8 | 4 5:$ $4 6 | 1 4,$ \newline
        $1 3 8 | 2 6 7 | 4 5:$ $4 6 | 1 4,$ \newline
    $1 6 7 | 2 3 8 | 4 5:$ $4 6 | 2 4,$ \newline
        $1 6 8 | 2 3 7 | 4 5:$ $4 6 | 2 4,$ \newline
    $1 7 8 | 2 3 6 | 4 5:$ $47|37|\ul{24},$ $47|37,$ $47|24,$ $37|24,$ \newline
$1 2 3 5 | 4 6 | 7 8:$ $47|24,$ \newline
    $1 2 3 | 4 6 | 5 7 8:$ $24|25,$ \newline
    $1 2 5 | 3 7 8 | 4 6:$ $47|14,$ \newline
    $1 3 5 | 2 7 8 | 4 6:$ $47|14,$ \newline
    $1 7 8 | 2 3 5 | 4 6:$ $47|24,$ \newline
        $1 2 3 8 | 4 7 | 5 6:$ $2 4 | 2 5,$ \newline
    $1 2 3 | 5 6 8 | 4 7:$ $2 4 | 2 5,$ \newline
        $1 5 6 | 2 3 8 | 4 7:$ $2 4 | 2 5,$ \newline
    $1 2 3 | 4 8 | 5 6 7:$ $2 4 | 2 5,$ $24|15,$ \newline
           $1 2 3 4 | 5 6 | 7 8:$ $57|35,$ $57|25,$ \newline
           $1 2 3 | 4 7 8 | 5 6:$ $57|24|\ul{25},$ $57|24,$ $57|25,$ $57|15,$ $24|25,$ \newline
$1 2 4 | 3 7 8 | 5 6:$ $57|25,$ $57|15,$ \newline
    $1 2 7 | 3 4 8 | 5 6:$ $35|13,$ \newline
    $1 2 8 | 3 4 7 | 5 6:$ $35|13,$ $35|15,$ \newline
$1 3 4 | 2 7 8 | 5 6:$ $57|35,$ $57|15,$ \newline
        $1 3 8 | 2 4 7 | 5 6:$ $25|15,$ \newline
    $1 4 7 | 2 3 8 | 5 6:$ $24|25|\ul{15},$ $24|25,$ $24|15,$ $25|15,$ \newline
        $1 4 8 | 2 3 7 | 5 6:$ $2 5 | 1 5,$ \newline
$1 7 8 | 2 3 4 | 5 6:$ $57|35,$ $57|25,$ \newline
        $1 2 6 | 3 4 8 | 5 7:$ $3 5 | 1 3,$ \newline
    $1 4 6 | 2 3 8 | 5 7:$ $24|15,$ \newline
        $1 2 3 4 | 5 8 | 6 7:$ $3 5 | 3 6,$ \newline
    $1 2 3 | 4 6 7 | 5 8:$ $2 5 | 1 4,$ \newline
        $1 2 4 | 3 6 7 | 5 8:$ $25|13,$ \newline
        $1 3 4 | 2 6 7 | 5 8:$ $3 5 | 3 6,$ \newline
        $1 6 7 | 2 3 4 | 5 8:$ $3 5 | 3 6,$ \newline
$1 2 3 | 4 5 8 | 6 7:$ $46|24,$ $46|26|\ul{14},$ $46|26,$ $46|14,$ $26|14,$ \newline
    $1 2 4 | 3 5 8 | 6 7:$ $36|26|\ul{13},$ $36|26,$ $36|13,$ $26|13,$ \newline
    $1 2 5 | 3 4 8 | 6 7:$ $36|26|\ul{13},$ $36|26,$ $36|13,$ $26|13,$ \newline
    $1 2 8 | 3 4 5 | 6 7:$ $46|13,$ \newline
    $1 3 4 | 2 5 8 | 6 7:$ $35|36|\ul{26},$ $35|36,$ $35|26,$ $36|26,$ \newline
$1 4 5 | 2 3 8 | 6 7:$ $46|24,$ $24|16,$ \newline
    $1 5 8 | 2 3 4 | 6 7:$ $35|36|\ul{16},$ $35|36,$ $35|16,$ $36|16,$ \newline
    $1 2 3 | 4 5 7 | 6 8:$ $2 6 | 1 4,$ \newline
        $1 2 4 | 3 5 7 | 6 8:$ $2 6 | 1 3,$ \newline
        $1 2 5 | 3 4 7 | 6 8:$ $2 6 | 1 3,$ \newline
        $1 2 7 | 3 4 5 | 6 8:$ $4 6 | 1 3,$ \newline
    $1 3 4 | 2 5 7 | 6 8:$ $35|26,$ \newline
    $1 4 5 | 2 3 7 | 6 8:$ $4 6 | 2 4,$ \newline
    $1 5 7 | 2 3 4 | 6 8:$ $3 5 | 1 6,$ \newline
$1 2 3 | 4 5 6 | 7 8:$ $57|24,$ $57|14,$ $27|14,$ \newline
    $1 2 4 | 3 5 6 | 7 8:$ $57|13,$ $27|13,$ \newline
    $1 2 5 | 3 4 6 | 7 8:$ $47|13,$ $27|13,$ \newline
    $1 2 6 | 3 4 5 | 7 8:$ $47|27|\ul{13},$ $47|27,$ $47|13,$ $27|13,$ \newline
$1 3 4 | 2 5 6 | 7 8:$ $57|35,$ $35|27,$ \newline
        $1 3 6 | 2 4 5 | 7 8:$ $47|37,$ \newline
$1 4 5 | 2 3 6 | 7 8:$ $47|37|\ul{24},$ $47|37,$ $47|24,$ $37|24,$ $24|17,$ \newline
$1 4 6 | 2 3 5 | 7 8:$ $47|24,$ \newline
           $1 5 6 | 2 3 4 | 7 8:$ $57|35,$ $57|25,$ $35|17,$ \newline
       $1 2 | 3 4 | 5 6 | 7 8:$ $57|35|13,$ \newline
           $1 2 | 3 4 | 5 8 | 6 7:$ $35|36|13,$ $35|36|16,$ \newline
           $1 2 | 3 5 | 4 6 | 7 8:$ $4 7 | 1 3 | 1 4$ \newline
           $1 2 | 3 6 | 4 5 | 7 8:$ $47|37|13|\ul{14},$ $47|37|13,$ $47|37|14,$ $47|13|14,$ $37|13|14,$ \newline
$1 2 | 3 7 | 4 5 | 6 8:$ $46|13|14,$ \newline
$1 2 | 3 8 | 4 5 | 6 7:$ $46|13|14,$ \newline
   $1 2 | 3 8 | 4 7 | 5 6:$ $13|14|15,$  \newline
$1 3 | 2 4 | 5 8 | 6 7:$ $25|26|16,$ \newline
$1 3 | 2 6 | 4 5 | 7 8:$ $47|27|14,$ \newline
           $1 4 | 2 3 | 5 6 | 7 8:$ $57|25|15,$ $25|15|17,$  \newline
$1 4 | 2 3 | 5 7 | 6 8:$ $25|15|16,$ \newline
$1 4 | 2 3 | 5 8 | 6 7:$ $25|26|15|\ul{16},$ $25|26|15,$ $25|26|16,$ $25|15|16,$ $26|15|16,$  \newline
   $1 5 | 2 3 | 4 8 | 6 7:$ $24|26|16,$ \newline
$1 6 | 2 3 | 4 5 | 7 8:$ $47|24|17,$ \newline
   $1 6 | 2 5 | 3 4 | 7 8:$ $37|27|17.$ \newline
\end{multicols}

\bibliographystyle{plain} 
\bibliography{upptrisim}

\end{document}